\documentclass{scrartcl}

\usepackage{amsmath,amsfonts,amsthm, authblk}%
\usepackage[linesnumbered,vlined,boxed]{algorithm2e}%
\let\oldnl\nl% Store \nl in \oldnl
\newcommand{\nonl}{\renewcommand{\nl}{\let\nl\oldnl}}% Remove line number for one line

\usepackage{xcolor}

\usepackage{enumerate}%
\usepackage{tikz}%
\usepackage{pgfplots}%

\usepackage[numbers,sort&compress]{natbib}% pour citer [1-5] a la place de [1,2,3,4,5] par exemple

\usepackage{hyperref}%
\hypersetup{% gestion des liens
  colorlinks=true,% colorise les liens
  breaklinks=true,% permet le retour a la ligne dans les liens trop longs
  urlcolor= blue,% couleur des hyperliens
  linkcolor= blue,% couleur des liens internes
  bookmarksopen=true,% si les signets Acrobat sont crees, les afficher completement.
  pdftitle={Equisingularity},% informations apparaissant dans
  pdfauthor={Martin Weimann},% dans les informations du document
}

\newcommand{\hbb}[1]{\ensuremath{\mathbb{#1}}}% anneaux / corps
\newcommand{\Ai}{\hbb{A}}
\newcommand{\Ci}{\hbb{C}}

\newcommand{\Ki}{\hbb{K}}
\newcommand{\Ni}{\hbb{N}}
\newcommand{\Qi}{\hbb{Q}}

\newcommand{\Li}{\hbb{L}}

\newcommand{\hcal}[1]{\ensuremath{\mathcal{#1}}}

\newcommand{\Bc}{\hcal{B}}

\newcommand{\Nc}{\hcal{N}}
\newcommand{\Oc}{\hcal{O}}

\theoremstyle{plain}% italique, extra space
\newtheorem{thm}{Theorem}
\newtheorem{prop}{Proposition}
\newtheorem{cor}{Corollary}
\newtheorem{lem}{Lemma}
\theoremstyle{definition}% roman, extra space
\newtheorem{dfn}{Definition}

\newtheorem{xmp}{Example}
\theoremstyle{remark}% roman, no space
\newtheorem{rem}{Remark}

\newcommand{\Edata}{\texttt{EdgeData}}

\newcommand{\PIrr}{\texttt{Pseudo-Irreducible}}

\newcommand{\AppRoot}{\texttt{AppRoot}}
\newcommand{\Expand}{\texttt{Expand}}

\renewcommand{\O}{\textrm{\Oc}}% O
\newcommand{\Ot}{\O\tilde\,\,}% soft-O
\newcommand{\dy}{{d}}% degre en y de F
\newcommand{\dx}{{n}}% degre en x de F, pour le paragraphe "bivariate case"

\newcommand{\val}[1][x]{\ensuremath{v_{#1}}}% valuation (en x, y etc)
\newcommand{\D}{\textrm{Data}}
\newcommand{\C}{\textrm{C}}
\newcommand{\Cont}{\textrm{Cont}}
\newcommand{\Res}{\textrm{Res}}
\newcommand{\coef}{\textrm{Coef}}

\newcommand{\Card}{\textrm{Card}}

\newcommand{\vF}[1][]{{\ensuremath{\delta_{#1}}}}% valuation du discriminant de F et F_y

\newcommand{\edgepoly}{boundary polynomial}
\newcommand{\True}{\texttt{True}}
\newcommand{\False}{\texttt{False}}
\newcommand{\Char}{\textrm{Char}}

\newcommand{\NP}{\Nc}% polygone de Newton
\newcommand{\algclos}[1]{\overline{#1}}% cloture algebrique
\begin{document}

\title{Computing the equisingularity type of a pseudo-irreducible polynomial}

\date{}
\author{%
  % \alignauthor
  Adrien Poteaux,\\%
  {CRIStAL, Universit\'e de Lille}\\%
  {UMR CNRS 9189, B\^atiment Esprit}\\%
  {59655 Villeneuve d'Ascq, France}\\%
  \texttt{adrien.poteaux@univ-lille.fr}
  % \alignauthor
  \and Martin Weimann\\%
  {GAATI\footnote{Current delegation. Permanent position at LMNO,
     University of Caen-Normandie, BP 5186, 14032 Caen Cedex,
     France.}
, Universit\'e de Polyn\'esie Fran\c{c}aise}\\%
  {UMR CNRS 6139, BP 6570}\\%
  {98702 Faa'a, Polyn\'esie Fran\c{c}aise}\\%
  \texttt{martin.weimann@upf.pf}%
}

% \author[1]{Adrien POTEAUX}
% \author[2]{Martin WEIMANN}
 
% \affil[1]{University of Lille, France}
% \affil[2]{University of French Polynesia, France}

\maketitle

\begin{abstract}
  Germs of plane curve singularities can be classified accordingly to
  their equisingularity type. For singularities over $\Ci$, this
  important data coincides with the topological class. In this paper,
  we characterise a family of singularities, containing irreducible
  ones, whose equisingularity type can be computed in quasi-linear
  time with respect to the discriminant valuation of a Weierstrass
  equation.
\end{abstract}

\section{Introduction}
Equisingularity is the main notion of equivalence for germs of plane
curves. It was developed in the 60's by Zariski over algebraically
closed fields of characteristic zero in \cite{ZaI65, ZaII65, Za68} and
generalised in arbitrary characteristic by Campillo \cite{Ca80}. This
concept is of particular importance as for complex curves, it agrees
with the topological equivalence class \cite{Za65}.  As illustrated by
an extensive litterature (see e.g. the book \cite{Gr07} and the
references therein), equisingularity plays nowadays an important role
in various active fields of singularity theory (resolution,
equinormalisable deformation, moduli problems, analytic
classification, etc). It is thus an important issue of computer
algebra to design efficient algorithms for computing the
equisingularity type of a singularity. This paper is dedicated to
characterise a family of reduced germs of plane curves, containing
irreducible ones, for which this task can be achieved in quasi-linear
time with respect to the discriminant valuation of a Weierstrass
equation.

\paragraph{Main result.}
We say that two germs of reduced plane curves are equisingular if there
is a one-to-one correspondance between their branches which preserves
the characteristic exponents and the pairwise intersection
multiplicities (see e.g. \cite{Ca80, Ca00,Wa04} for other equivalent definitions). This equivalence relation leads to the notion of equisingularity type of a singularity. In this paper, we consider a square-free Weierstrass polynomial
$F\in\Ki[[x]][y]$ of degree $\dy$, with $\Ki$ a perfect field of
characteristic zero or greater than $\dy$\footnote{Our results still
  hold under the weaker assumption that the characteristic of $\Ki$
  does not divide $\dy$.}. Under such assumption, the Puiseux series of $F$
are well defined and allow to determine the equisingularity type of
the germ $(F,0)$ (the case of small characteristic requires
Hamburger-Noether expansions \cite{Ca80}). In particular, it follows
from \cite{PoWe17} that we can compute the equisingularity type in an
expected $\Ot(\dy\,\vF)$ operations over $\Ki$, where $\vF$ stands for the valuation of the
discriminant of $F$. If moreover $F$ is
irreducible, it is shown in \cite{PoWe19} that we can reach the lower
complexity $\Ot(\vF)$ thanks to the theory of approximate roots. In
this paper, we extend this result to a larger class of polynomials.

We say that $F$ is \emph{balanced} or
\textit{pseudo-irreducible}\footnote{In the sequel, we rather use
  first the terminology \emph{balanced} and give an alternative
  definition of pseudo-irreducibility based on a Newton-Puiseux type
  algorithm. Both notions agree from Theorem \ref{thm:pseudo}.} if
all its absolutely irreducible factors have the same set of
characteristic exponents and the same set of pairwise intersection
multiplicities, see Section \ref{sec:equising}. Irreducibility over
any algebraic field extension of $\Ki$ implies pseudo-irreducibility
by a Galois argument, but the converse does not hold. As a basic
example, the Weierstrass polynomial $F=(y-x)(y-x^2)$ is
pseudo-irreducible, but is obviously reducible. We prove:

\begin{thm}\label{thm:main2}
  There exists an algorithm which tests if $F$ is pseudo-irreducible
  with an expected $\Ot(\vF)$\footnote{As usual, the notation $\Ot()$
    hides logarithmic factors. Note that $F$ being Weierstrass, we
    have $\dy\le \vF$ and $\vF\log(\dy)\in\Ot(\vF)$.} operations over
  $\Ki$. If $F$ is pseudo-irreducible, the algorithm computes also
  $\vF$ and the number of absolutely irreducible factors of $F$
  together with their sets of characteristic exponents and sets of
  pairwise intersection multiplicities. In particular, it computes the
  equisingularity type of the germ $(F,0)$.
\end{thm}

The algorithm contains a Las Vegas subroutine for computing primitive
elements in residue rings; however it should become deterministic
thanks to the recent preprint \cite{HoLe18}. For a given field
extension $\Li$ of $\Ki$, we can also compute the degrees, residual
degrees and ramification indices of the irreducible factors of $F$ in
$\Li[[x]][y]$ by performing an extra univariate factorisation of
degree at most $\dy$ over $\Li$.  Having a view towards fast
factorisation in $\Ki[[x]][y]$, we can extend the definition of
pseudo-irreducibility to non Weierstrass polynomials, taking into
account all germs of curves defined by $F$ along the line $x=0$. Our
approach adapts to this more general setting, with complexity
$\Ot(\vF+\dy)$\footnote{When $F$ is not Weierstrass, we might have
  $\dy\notin\Ot(\vF)$.} (see Section \ref{sec:psi}).

\paragraph{Main tools.} We generalise the irreducibility test obtained
in \cite{PoWe19}, which is itself a generalisation of Abhyankhar's
absolute irreducibility criterion \cite{Ab89}, based on the theory of
approximate roots. The main idea is to compute recursively some
suitable approximate roots $\psi_0,\ldots,\psi_g$ of $F$ of strictly
increasing degrees such that $F$ is pseudo-irreducible if and only if
we reach $\psi_g=F$. At step $k$, we compute the
$(\psi_0,\ldots,\psi_k)$-adic expansion of $F$ from which we can
construct a generalised Newton polygon. If the corresponding
\edgepoly{} of $F$ is not pseudo-degenerated (Definition
\ref{def:pseudodeg}), then $F$ is not pseudo-irreducible. Otherwise,
we deduce the degree of the next approximate root $\psi_{k+1}$ that
has to be computed.

The key difference when compared to the irreducibility test developped
in \cite{PoWe19} is that we may allow the successive generalised
Newton polygons to have several edges, although no splittings and no
Hensel liftings are required. Except this slight modification, most of
the algorithmic considerations have already been studied in
\cite{PoWe19} and this paper is more of a theoretical nature, focused
on two main points : proving that pseudo-degeneracy is the right
condition for characterising pseudo-irreducibility, and giving
formulas for the intersection multiplicities and characteristic
exponents in terms of the underlying edge data sequence. This is our
main Theorem \ref{thm:pseudo}.

\paragraph{Related results.} Computing the equisingularity type of a
plane curve singularity is a classical topic for which both symbolic
and numerical methods exist. A classical approach is derived from the
Newton-Puiseux algorithm, as a combination of blow-ups (monomial
transforms and shifts) and Hensel liftings. This approach allows to
compute the roots of $F$ - represented as fractional Puiseux series -
up to an arbitrary precision, from which the equisingularity type of
the germ $(F,0)$ can be deduced (see e.g. Theorem \ref{thm:pseudo} for
precise formulas). The Newton-Puiseux algorithm has been studied by
many authors (see e.g. \cite{Du89, DeDiDu85, Wa00, Te90, Po08, PoRy11,
  PoRy15, PoWe17, PoWe19} and the references therein). Up to our
knowledge, the best current arithmetic complexity was obtained in
\cite{PoWe17}, computing the singular parts of all Puiseux series
above $x=0$ - hence the equisingularity type of all germs of curves
defined by $F$ along this line - in an expected $\Ot(\dy\,\vF)$
operations over $\Ki$. Here, we get rid of the $\dy$ factor for
pseudo-irreducible polynomials, generalising the irreducible case
considered in \cite{PoWe19}.  For complex curves, the equisingularity
type agrees with the topological class and there exists other
numerical-symbolic methods of a more topological nature (see
e.g. \cite{HoMoSc10,HoMoSc11,HoNgPh17,HoSc11,SoWa05} and the
references therein).  This paper comes from a longer preprint
\cite{PoWeBis19} which contains also results of \cite{PoWe19}.

\paragraph{Organisation.} We define balanced polynomials in Section
\ref{sec:equising}. Section \ref{sec:pseudo-irr} introduces the notion
of pseudo-degeneracy. This leads to an alternative definition of a
pseudo-irreducible polynomial, based on a Newton-Puiseux type
algorithm. In Section \ref{sec:equivalent}, we prove that being
balanced is equivalent to being pseudo-irreducible and we give
explicit formulas for characteristic exponents and intersection
multiplicities in terms of edge data (Theorem \ref{thm:pseudo}). In
the last Section \ref{sec:psi}, we design a pseudo-irreducibility test
based on approximate roots with quasi-linear complexity, thus proving
Theorem \ref{thm:main2}. We illustrate our method on various examples.

%%% Local Variables:
%%% mode: latex
%%% TeX-master: "equi-singularity"
%%% End:

\section{Balanced polynomials}\label{sec:equising}
Let us fix $F\in \Ki[[x]][y]$ a Weierstrass polynomial defined over a
perfect field $\Ki$ of characteristic zero or greater than
$\dy=\deg(F)$. For simplicity, we abusively denote by
$(F,0)\subset (\algclos{\Ki}^2,0)$ the germ of the plane curve defined
by $F$ at the origin of the affine plane $\algclos{\Ki}^2$. We say
that $F$ is absolutely irreducible if it is irreducible in
$\algclos{\Ki}[[x]][y]$. The germs of curves defined by the absolutely
irreducible factors of $F$ are called the branches of the germ
$(F,0)$.

\subsection{Characteristic exponents}\label{ssec:charexp}
We assume here that $F$ is absolutely irreducible. As the
characteristic of $\Ki$ does not divide $\dy$, there exists a unique
series $S(T)=\sum c_i T^i\in \algclos{\Ki}[[T]]$ such that $F(T^d,S(T))=0$. The pair $(T^d, S(T))$ is the
classical Puiseux parametrisation of the branch $(F,0)$. The
\textit{characteristic exponents} of $F$ are defined as
\[
\beta_0=d,\quad \beta_k=\min\left(i\,\,\text{ s.t.} \,\, c_i\ne 0,\,
  gcd(\beta_0,\ldots,\beta_{k-1}) \not| i \right), \quad k=1,\ldots,g,
\]
where $g$ is the least integer for which
$gcd(\beta_0,\ldots,\beta_g)=1$ (characteristic exponents are
sometimes refered to the rational numbers $\beta_i/d$ in the
litterature). These are the exponents $i$ for which a non trivial
factor of the ramification index is discovered. It is well known that
the data
\[
\C(F)=(\beta_0;\beta_1,\ldots,\beta_g)
\]
determines the equisingularity type of the germ $(F,0)$, see e.g.
\cite{Za65}. Conversely, the Weierstrass equations of two equisingular
germs of curves which are not tangent to the $x$-axis have same
characteristic exponents \cite[Corollary 5.5.4]{Ca00}. If tangency
occurs, we rather need to consider the characteristic exponents of the
local equation obtained after a generic change of local coordinates,
which form a complete set of equisingular (hence topological if
$\Ki=\Ci$) invariants. The set $\C(F)$ and the set of generic
characteristic exponents determine each others assuming that we are
given the contact order  $\beta_0$ with $x$-axis (\cite[Proposition
4.3]{Pop02} or \cite[Corollary 5.6.2]{Ca00}). It is well known that a
data equivalent to $\C(F)$ is given by the semi-group of $F$, and that
this semi-group admits the intersection multiplicities of $F$ with its
characteristic approximate roots $\psi_{-1},\psi_0,\ldots,\psi_g$ as a
minimal system of generators (see Section \ref{ssec:phi-main} and
\cite[Corollaries 5.8.5 and 5.9.11]{Ca00}).

\subsection{Intersection sets}\label{ssec:IntSets}

If we want to determine the equisingularity type of a reducible germ
$(F,0)$, we need to consider also the pairwise intersection
multiplicities between the absolutely irreducible factors of $F$. The
intersection multiplicity between two coprime Weierstrass polynomials
$G,H\in \Ki[[x]][y]$ is defined as
\begin{equation}\label{eq:intmult}
  (G,H)_0:=\val(\Res_y(G,H))=\dim_{\algclos{\Ki}}\frac{\algclos{\Ki}[[x]][y]}{(G,H)},
\end{equation}
where $\Res_y$ stands for the resultant with respect to $y$ and $\val$
is the usual $x$-valuation. The right hand equality follows from
classical properties of the resultant. Suppose that $F$ has (distinct)
absolutely irreducible factors $F_1,\ldots,F_{f}$. We introduce
\textit{the intersection sets} of $F$, defined for $i=1,\ldots,f$ as
\[
\Gamma_i(F):=\left((F_i,F_j)_0,\, 1\le j \le f, j\ne i\right).
\]
By convention, we take into account repetitions, $\Gamma_i(F)$ being
considered as an unordered list with cardinality $f-1$.  If $F$ is
Weierstrass, the equisingular type (hence the topological class if
$\Ki=\Ci$) of the germ $(F,0)$ is uniquely determined by the
characteristic exponents and the intersections sets of the branches of
$F$ \cite{Za73}. Note that the set $\C(F_i)$ only depends on $F_i$
while $\Gamma_i(F)$ depends on $F$.

\subsection{Balanced polynomials}\label{ssec:defequising}

\begin{dfn}\label{def:balanced}
  We say that a square-free Weierstrass\footnote{We can extend this
    definition to non Weierstrass polynomials, see Subsection
    \ref{ssec:nonweirestrass}.} polynomial $F\in \Ki[[x]][y]$ is
  \emph{balanced} if $\C(F_i)=\C(F_j)$ and $\Gamma_i(F)=\Gamma_j(F)$
  for all $i,j$. In such a case, we denote simply these sets by
  $\C(F)$ and $\Gamma(F)$.
\end{dfn}

Thus, if $F$ is balanced, its branches are equisingular and have the same
set of pairwise intersection multiplicities. The converse holds if no
branch is tangent to the $x$-axis or all branches are tangent to the
$x$-axis.

\begin{xmp}\label{rem:deg}
  Let us illustrate this definition with some basic examples. Note
  that the second and third examples show in particular that no
  condition implies the other in Definition \ref{def:balanced}.
  \begin{enumerate}
  \item If $F\in \Ki[[x]][y]$ is irreducible, a Galois argument shows
    that it is balanced (follows from Theorem \ref{thm:pseudo}
    below). The converse doesn't hold: $F=(y-x)(y+x^2)$ is reducible,
    but it is balanced. This example also shows that being balanced does
    not imply the Newton polygon to be straight.
  \item $F=(y^2-x^3)(y^2+x^3)(y^2+x^3+x^4)$ is not balanced. It has
    $3$ absolutely irreducible factors with same sets of
    characteristic exponents $\C(F_i)=(2;3)$ for all $i$, but
    $\Gamma_1(F)=(6,6)$ while $\Gamma_2(F)=\Gamma_3(F)=(6,8)$.
  \item $F=(y-x-x^2)(y-x+x^2)(y^2-x^3)$ is not balanced. It has $3$
    absolutely irreducible factors with same sets of pairwise
    intersection multiplicities $\Gamma_i(F)=(2,2)$, but
    $\C(F_1)=\C(F_2)=(1)$ while $\C(F_3)=(2;3)$.
  \item $F=(y^2-x^2)^2-2x^4y^2-2x^6+x^8$ has four absolutely
    irreducible factors, namely $F_1=y+x+x^2 $, $F_2=y+x-x^2$,
    $F_3=y-x+x^2$ and $F_4=y-x-x^2$.  We have $C(F_i)=(1)$ and
    $\Gamma_i(F)=(1,1,2)$ for all $i$ so $F$ is balanced. Note that
    this example shows that being balanced does not imply that all
    factors intersect each others with the same multiplicity.
  \item $F=(y^2-x^3)(y^3-x^2)$ is not balanced. However, it defines
    two equisingular germs of plane curves (but one is tangent to the
    $x$-axis while the other is not).
  \end{enumerate}
\end{xmp}

\paragraph{Noether-Merle's Formula.} If $F,G\in \algclos{\Ki}[[x]][y]$
are two irreducible Weierstrass polynomials of respective degrees
$d_F$ and $d_G$, their intersection multiplicty $(F,G)_0$ is closely
related to the characteristic exponents $(\beta_0,\ldots,\beta_g)$ of
$F$. Let us denote by 
\begin{equation}\label{eq:ContFormula}
\Cont(F,G):=d_F\,\max\left(\val(y-y') \,\, | \,\,F(y)=0,\,\,G(y')=0\right)
\end{equation}
the \textit{contact order} of the branches $F$ and $G$ and let
$\kappa = \max\{k\,|\,\Cont(F,G) \geq \beta_k\}$. Then Noether-Merle's
formula \cite[Proposition 2.4]{Me77}  states
\begin{equation}\label{eq:NMformula}
  (F,G)_0=\frac{d_G}{d_F}\left(\sum_{k\le\kappa} (E_{k-1}-E_{k})\beta_k+E_\kappa\,\Cont(F,G)\right),
\end{equation}
where $E_k:=gcd(\beta_0,\ldots,\beta_k)$. A proof can be found in
\cite[Proposition 6.5]{Pop02} (and references therein), where a
formula is given in terms of the semi-group generators, which turns
out to be equivalent to \eqref{eq:NMformula} thanks to
\cite[Proposition 4.2]{Pop02}. Note that the original proof in
\cite{Me77} assumes that the germs $F$ and $G$ are transverse to the
$x$-axis.

%%% Local Variables:
%%% mode: latex
%%% TeX-master: "equi-singularity"
%%% End:

\section{Pseudo-irreducible polynomials}\label{sec:pseudo-irr}
\subsection{Pseudo-degenerated polynomials.}\label{ssec:pseudodeg}
We first recall classical definitions that play a central role for our
purpose, namely the Newton polygon and the residual polynomial. We
will have to work over various residue rings isomorphic to some direct
product of fields extension of the base field $\Ki$. Let
$\Ai=\Li_0\oplus \cdots \oplus \Li_r$ be such a ring. If
$S=\sum c_i x^i\in \Ai[[x]]$, we define
$\val(S)=\min(i\,,\, c_i\ne 0)$ with convention
$\val(0)=+\infty$. Note that in contrast to usual valuations, we have
$\val(S_1\,S_2)\ge \val(S_1)+\val(S_2)$ and strict inequality might occur since
$\Ai$ is allowed to contain zero divisors.

In the following definitions, we assume that $F\in \Ai[[x]][y]$
is a Weierstrass polynomial and we let
$F=\sum_{i=0}^\dy a_i(x)\,y^i=\sum_{i,j} a_{ij}x^j y^i$. 
\begin{dfn}\label{dfn:NP}
  The \textit{Newton polygon} of $F$ is the lower convex hull $\NP(F)$
  of the set of points $(i,\val(a_i))$ with $a_i\ne 0$ and
  $i=0,\ldots,\dy$. We denote by $\NP_0(F)$ the lower edge (right hand
  edge) of the Newton polygon.
\end{dfn}
The lower edge has equation $mi+qj=l$ for some uniquely determined
coprime positive integers $q,m$ and $l\in \Ni$. We say for short that
$\NP_0(F)$ has slope $(q,m)$, with convention $(q,m)=(1,0)$ if the
Newton polygon of $F$ is reduced to a point.
\begin{dfn}\label{dfn:edgePol}
  We call $\bar{F}:=\sum_{(i,j)\in \NP_0(F)} a_{ij} x^j y^i$ the
  \emph{lower \edgepoly{}} of $F$.
\end{dfn}
We say that a polynomial $P\in \Ai[Z]$ is
\textit{square-free} if its images under the natural
morphisms $\Ai\to \Li_i$ are square-free (in the usual sense over a
field).

\begin{dfn}\label{def:pseudodeg}
  We say that $F\in\Ai[[x]][y]$ is \emph{pseudo-degenerated} if there
  exists $N\in \Ni$ and $P\in \Ai[Z]$ monic and square-free such that
\begin{equation}\label{eq:pseudoquasihom}
  \bar{F}=\left(P\left(\frac{y^q}{x^{m}}\right)x^{m\deg(P)}\right)^N,
\end{equation}
with moreover $P(0)\in \Ai^\times$ (units of $\Ai$) if $q>1$. We call
$P$ the \emph{residual polynomial} of $F$. The tuple $(q,m,P,N)$ is
the \emph{edge data} of $F$.
\end{dfn}

\begin{rem}\label{rem:testpseudodeg}
  In practice, we check pseudo-degeneracy as follows. If $q$ does not
  divide $\dy$, then $F$ is not pseudo-degenerated. If $q$ divides $\dy$,
  then $q|i$ for all $(i,j)\in \NP_0(F)$ as $(d,0)\in \NP_0(F)$ by
  assumption. Hence we may consider
  $Q=\sum_{(i,j)\in \NP_0(F)} a_{ij} Z^{i/q}\in \Ai[Z]$ and $F$ is
  pseudo-degenerated if and only if $Q=P^N$ for some square-free
  polynomial $P$ such that $P(0)\in \Ai^\times$ if $q>1$.
\end{rem}

\begin{rem}\label{rem:notstraight}
  If $q>1$, the extra condition $P(0)\in \Ai^\times$ implies that
  $\NP(F)$ is straight.  If $q=1$, we allow $P(0)$ to be a
  zero-divisor (in contrast to Definition 4 of quasi-degeneracy in
  \cite{PoWe19}), in which case $\NP(F)$ may have several edges. Note
  that if $F$ is pseudo-degenerated, $\bar{F}$ is the power of a
  square-free quasi-homogeneous polynomial, but the converse doesn't
  hold (case \ref{it:xmp-pseudo4} below).
\end{rem}

\begin{xmp}{\textcolor{white} a}
  \begin{enumerate}
  \item Let $F=(y^2-x^2)^2(y-x^2)(y-x^3)$. Then $\NP(F)$ has three
    edges, the lower one of slope $(q,m)=(1,1)$. We get
    $\bar{F}=(y^3-x^2y)^2$ and $Q=(Z^3-Z)^2$. Hence, $F$ is
    pseudo-degenerated, with $P=Z^3-Z$ and $N=2$.
  \item Let $F=(y^2-x^2)^2(y-x^2)$. Then $\NP(F)$ has two edges, the
    lower one of slope $(q,m)=(1,1)$. We get $\bar{F}=y(y^2-x^2)^2$
    and $Q=Z(Z^2-1)^2$ is not a power of a square-free
    polynomial. Hence, $F$ is not pseudo-degenerated.
  \item Let $F=(y^2-x^3)^2(y-x^4)$. Then $\NP(F)$ has two edges, the
    lower one of slope $(q,m)=(2,3)$. As $q$ does not divide $\dy=5$,
    $F$ is not pseudo-degenerated.
  \item \label{it:xmp-pseudo4} Let $F=(y^2-x^3)^2(y-x^4)^2$. Then
    $\NP(F)$ is straight of slope $(q,m)=(2,3)$. Here $q$ divides
    $\dy=6$. We get $\bar{F}=y^2(y^2-x^3)^2$ which is a power of a
    square-free polynomial. However, $Q=Z(Z-1)^2$ is not. Hence, $F$
    is not pseudo-degenerated.
  \item Let $F=(y^2-x^3)^2(y^2-x^4)^2$. Then $\NP(F)$ has two edges,
    the lower one of slope $(q,m)=(2,3)$. Here $q$ divides $\dy=8$. We
    get $\bar{F}=(y^4-y^2x^3)^2$ and $Q=(Z^2-Z)^2$ is the power of the
    square-free polynomial $P=Z^2-Z$. However, $q>1$ and $P(0)=0$ so
    $F$ is not pseudo-degenerated.
  \end{enumerate}
\end{xmp}
Note that we could also treat cases 3, 4 and 5 simply by using Remark
\ref{rem:notstraight}: $q>1$ and $\NP(F)$ not straight imply that $F$
is not pseudo-degenerated.

The next lemma allows to associate to a pseudo-degenerated polynomial
$F$ a new Weierstrass polynomial of smaller degree, generalising the
usual case (e.g. \cite[Sec.4]{Du89} or \cite[Prop.3]{PoRy15}) to the case of product of fields.

\begin{lem}\label{lem:anptransform}
  Suppose that $F$ is pseudo-degenerated with edge data $(q,m,P,N)$
  and denote $(s,t)$ the unique positive integers such that
  $s\,q-t\,m=1$, $0\le t < q$. Let $z$ be the residue class of $Z$ in
  the ring $\Ai_P:=\Ai[Z]/(P(Z))$ and $\ell:=\deg(P)$. Then
  \begin{equation}\label{eq:anptransform}
    F(z^{t} x^{q},x^{m}(y+z^{s}))=x^{qm\ell N} U G,
  \end{equation}
  where $U,G\in \Ai_P[[x]][y]$, $U(0,0)\in \Ai_P^\times$ and $G$ is a
  Weierstrass polynomial of degree $N$ dividing $\dy$. Moreover, if
  $F\ne y^\dy$ and $F$ has no terms of degree $\dy-1$, then $N<\dy$.
\end{lem}
\begin{proof}
  Let $\tilde{F}(x,y)=F(z^{t} x^{q},x^{m}(y+z^{s}))x^{-qm\ell N}$. We
  deduce from \eqref{eq:pseudoquasihom} that
  $\tilde{F}\in \Ai_P[[x]][y]$ and $\tilde{F}(0,y)=R(y)^N$ where
  $R(y)=P((y+z^s)^q/z^{tm})$. We have $R(0)=P(z)=0$ while
  $R'(0)=q z^{1-s}P'(0)$. As $P$ is square-free and the characteristic
  of $\Ai$ does not divide $\deg(P)$ by assumption, we have
  $P'(0)\in \Ai^\times$. As $q z^{1-s}\in \Ai_P^\times$ (if $q>1$, the
  assumption $P(0)\in \Ai^\times$ implies $P$ and $Z$ coprime, that is
  $z\in\Ai_P^\times$; if $q=1$, then $s=1$), it follows that
  $R'(0)\in \Ai_P^\times$. We deduce that $\tilde{F}(0,y)=y^N S(y)$
  where $y^N$ and $S(y)$ are coprime in $\Ai_P[y]$. We conclude thanks
  to the Weierstrass preparation theorem that $F$ factorises as in
  \eqref{eq:anptransform}. Note that $N|\dy$ by
  \eqref{eq:pseudoquasihom}. If $N=\dy$, then
  \eqref{eq:pseudoquasihom} forces $\bar{F}=(y+\alpha x^m)^\dy$ for
  some $\alpha\in \Ai$. If $\alpha = 0$, then $\NP(F)$ is reduced to a
  point and we must have $F=y^\dy$. If $\alpha \ne 0$, the coefficient
  of $y^{\dy-1}$ in $F$ is $\dy\alpha x^{m\,\dy}+h.o.t$, hence is non zero
  since the characteristic of $\Ai$ does not divide $\dy$.
\end{proof}

\begin{rem}\label{rem:q=1}
  As $P\in \Ai[Z]$ is square-free, the ring $\Ai_P=\Ai[Z]/(P(Z))$ is still
  isomorphic to a direct product of perfect fields thanks to the
  Chinese Reminder Theorem. Note also that $z^t$ is invertible: if
  $z$ is a zero divisor, we must have $q=1$ so that $t=0$ and $z^t=1$.
\end{rem}

\subsection{Pseudo-irreducible polynomials}\label{ssec:pseudoIrr}
The definition of a pseudo-irreducible polynomial is based on a
variation of the classical Newton-Puiseux algorithm.  Thanks to
Lemma \ref{lem:anptransform}, we associate to $F$ a sequence a
Weierstrass polynomials $H_0,\ldots,H_g$ of strictly decreasing
degrees $N_0,\ldots,N_g$ such that $H_k$ is pseudo-degenerated if
$k<g$ and such that either $H_g$ is not pseudo-degenerated either
$N_g=1$. We proceed recursively as follows:

$\bullet $ \textbf{Rank $k=0$}. Let $N_0=\dy$, $\Ki_0=\Ki$,
$c_0(x):=-\coef(F,y^{N_0-1})/N_0$ and
\begin{equation}\label{eq:H0}
  H_0(x,y):=F(x,y+c_0(x))\in \Ki_0[[x]][y].
\end{equation}
Then $H_0$ is a new Weierstrass polynomial of degree $N_0$ with no
terms of degree $N_0-1$. If $N_0=1$ or $H_0$ is not
pseudo-degenerated, we let $g=0$.

$\bullet $ \textbf{Rank $k>0$}. Suppose given $\Ki_{k-1}$ a direct
product of fields extension of $\Ki$ and
$H_{k-1}\in \Ki_{k-1}[[x]][y]$ a pseudo-degenerated Weierstrass
polynomial of degree $N_{k-1}>1$, with no terms of degree
$N_{k-1}-1$. Denote by $(q_k,m_k,P_k,N_k)$ its edge data and
$\ell_k:=\deg(P_k)$. As $P_k$ is square-free, the ring
$\Ki_k:=\Ki_{k-1}[Z_k]/(P_k(Z_k))$ is again (isomorphic to) a direct
product of fields. We let $z_k\in \Ki_k$ be the residue class of $Z_k$
and $(s_k,t_k)$ the unique positive integers such that
$s_k\,q_k-t_k\,m_k=1$, $0\le t_k < q_k$. As $H_{k-1}$ is
pseudo-degenerated, we deduce from Lemma \ref{lem:anptransform} that
\begin{equation}\label{eq:Gk}
  H_{k-1}(z_k^{t_k} x^{q_k},x^{m_k}(y+z_k^{s_k}))=x^{q_km_k\ell_k N_k} V_k G_k,
\end{equation}
where $V_k(0,0)\in \Ki_k^\times$ and $G_k\in \Ki_k[[x]][y]$ is a
Weierstrass polynomial of degree $N_k$.
Letting $c_k:=-\coef(G_k,y^{N_k-1})/N_{k}$, we define
\begin{equation}\label{eq:Hk}
  H_k(x,y)=G_k(x,y+c_k(x))\in \Ki_k[[x]][y].
\end{equation}
It is a degree $N_k$ Weierstrass polynomial with no terms of degree
$N_k-1$.

$\bullet$ \textbf{The $N_k$-sequence stops.} We have the relations
$N_k=q_k \ell_k N_{k-1}$. As $H_{k-1}$ is pseudo-degenerated with no
terms of degree $N_{k-1}-1$, we have $N_{k}<N_{k-1}$ by Lemma \ref{lem:anptransform}. Hence the sequence of integers $N_0,\ldots,N_k$
is strictly decreasing and there exists a smallest index $g$ such that
either $N_g=1$ (and $H_g=y$), either $N_g>1$ and $H_g$ is not
pseudo-degenerated. We collect the edge data of the polynomials
$H_0,\ldots,H_{g-1}$ in a list
\[
\D(F):=\left((q_1,m_1,P_1,N_1),\ldots,(q_{g},m_{g},P_{g},N_{g})\right).
\]
Note that $m_{k}>0$ for all $1\leq k \le g$. We include the $N_k$'s in
the list for convenience (they could be deduced from the remaining
data via the relations $N_{k}=N_{k-1}/ q_k\ell_k$).
\begin{dfn}\label{def:pseudo-irr}
  We say that $F$ is \emph{pseudo-irreducible} if $N_g=1$.
\end{dfn}

%%% Local Variables:
%%% mode: latex
%%% TeX-master: "equi-singularity"
%%% End:

\section{Pseudo-irreducible is equivalent to
  balanced.}\label{sec:equivalent}
We prove here our main result, Theorem \ref{thm:pseudo}: a square-free
Weierstrass polynomial $F\in \Ki[[x]][y]$ is pseudo-irreducible if and
only if it is balanced, in which case we compute characteristic
exponents and intersection sets of the irreducible factors.
\subsection{Notations and main results.}\label{ssec:notations}
We keep notations of Section \ref{sec:pseudo-irr}; in particular
$(q_1,m_1,P_1,N_1),\ldots,(q_g,m_g,P_g,N_g)$ denote the edge data of
$F$. We define $e_k:=q_1\cdots q_k$ (current index of ramification),
$e:=e_g$, $\hat{e}_{k}:=e/e_k$ and in an analoguous way
$f_k:=\ell_1\cdots \ell_k$ (current residual degree), $f:=f_g$ and
$\hat{f}_k:=f/f_k$. For all $k=1,\ldots,g$, we define
\begin{equation}\label{eq:BkMk}
  B_k=m_1 \hat{e}_1+\cdots +m_k\hat{e}_k
  \quad \textrm{and} \quad
  M_k=m_1 \hat{e}_0\hat{e_1}+\cdots+ m_k\hat{e}_{k-1}\hat{e}_k
\end{equation}
and we let $B_0=e$. These are positive integers related by the formula
\begin{equation}\label{eq:Mk}
  M_k=\sum_{i=1}^k(\hat{e}_{i-1}-\hat{e}_i)B_i +\hat{e}_k B_k.
\end{equation}
Note that $0 < B_1\le\cdots \le B_g$\footnote{We may allow $m_1=B_1=0$
  when considering non Weierstrass polynomials, see Subsection
  \ref{ssec:nonweirestrass}.} and $B_0\le B_g$. We have $B_0\le B_1$
if and only if $q_1\le m_1$, if and only if $F=0$ is not tangent to
the $x$-axis at the origin. We check easily that
$\hat{e}_k=gcd(B_0,\ldots,B_k)$. In particular,
$gcd(B_0,\ldots,B_g)=1$.
\begin{thm}\label{thm:pseudo}
  A Weierstrass polynomial $F\in \Ki[[x]][y]$ is balanced if and only
  if it is pseudo-irreducible. It such a case, $F$ has $f$ irreducible
  factors in $\algclos{\Ki}[[x]][y]$, all with degree $e$, and
  \begin{enumerate}
  \item $\C(F)=(B_0;B_k \, |\, q_k >1)$ - so $\C(F)=(1)$ if $q_k=1$ for
    all $k$.
  \item $\Gamma(F)=(M_k \, |\, \ell_k >1)$, where $M_k$ appears
    $\hat{f}_{k-1}-\hat{f}_{k}$ times.
  \end{enumerate}
\end{thm}
Taking into account repetitions, the intersection set has cardinality
$\sum_{k=1}^g (\hat{f}_{k-1}-\hat{f}_{k})=f-1$, as required. Of
course, it is empty if and only if $F$ is absolutely
irreducible.
\begin{cor}\label{cor:disc}
  Let $F\in \Ki[[x]][y]$ be a balanced Weierstrass polynomial. Then,
  the discriminant of $F$ has valuation
  \[
  \vF= f \left(\sum_{\ell_k > 1}
  (\hat{f}_{k-1}-\hat{f}_{k})M_k+\sum_{q_k >1}
  (\hat{e}_{k-1}-\hat{e}_{k})B_k\right)
  \]
  and the discriminants of the absolutely irreducible factors of $F$
  all have the same valuation
  $\sum_{q_k >1} (\hat{e}_{k-1}-\hat{e}_{k})B_k$.
\end{cor}
\begin{proof} \textit{$($of Corollary \ref{cor:disc}$)$} When $F$
  is balanced, it has $f$ irreducible factors $F_1,\ldots,F_f$ of
  same degree $e$, with discriminant valuations say
  $\vF_1,\ldots,\vF_f$. The multiplicative property of the
  discriminant gives the well-known formula
  \begin{equation}\label{eq:disc}
    \vF=\sum_{1\le i \le f} \vF_i + \sum_{1\le i\ne j\le f}(F_i,F_j)_0.
  \end{equation}
  Let $y_1,\ldots,y_e$ be the roots of $F_i$. Thanks to
  \cite[Proposition 4.1.3 (ii)]{Wa04} combined with point 1 of Theorem
  \ref{thm:pseudo}, we deduce that for each fixed $a=1,\ldots,e$, the
  list $(\val(y_a-y_b),b\ne a)$ consists of the values $B_k/e$
  repeated $\hat{e}_{k-1}-\hat{e}_{k}$ times for $k=1,\ldots,g$. Since
  $\vF_i=\sum_{1\le a \ne b \le e} \val(y_a-y_b)$, we deduce that
  $\vF_1=\cdots=\vF_f=\sum_{q_k >1} (\hat{e}_{k-1}-\hat{e}_{k})B_k$. The
  formula for $\vF$ follows directly from
  \eqref{eq:disc} combined with point 2 of Theorem \ref{thm:pseudo}.
\end{proof}
The remaining part of this section is dedicated to the proof of
Theorem \ref{thm:pseudo}. It is quite technical, but has the advantage
to be self-contained. We first establish the relations between the
(pseudo)-rational Puiseux expansions and the classical Puiseux series
of $F$ (Subsection \ref{ssec:pseudoRNP}). This allows us to deduce the
characteristic exponents and the intersection sets of a
pseudo-irreducible polynomial (thanks to Noether-Merle's formula),
proving in particular that pseudo-irreducible implies balanced
(Subsection \ref{ssec:PIrrToEqSing}).  We prove the more delicate
reverse implication in Subsection \ref{ssec:EqSingToPIrr}.

\subsection{Pseudo-rational Puiseux expansion.}\label{ssec:pseudoRNP}
Keeping notations of Section \ref{sec:pseudo-irr}, let
$\pi_0(x,y)=(x,y+c_0(x))$ and $\pi_k=\pi_{k-1}\circ \sigma_k$ where
\begin{equation}\label{eq:tausigma}
  \sigma_k(x,y):=(z_k^{t_k} x^{q_k},x^{m_k}(y+z_k^{s_k}+c_k(x)))
\end{equation}
for $k\ge 1$. It follows from equalities \eqref{eq:H0}, \eqref{eq:Gk}
and \eqref{eq:Hk} that
\begin{equation}\label{eq:pikHk}
\pi_k^* F= U_k H_k\in \Ki_k[[x,y]]
\end{equation}
for some $U_k$ such that $U_k(0,0)\in \Ki_k^\times$. We deduce from
\eqref{eq:tausigma} that
\begin{equation}\label{eq:pikxy}
  \pi_k (x, y ) = (\mu_k x^{e_k},\alpha_k x^{r_k} y + S_{k} (x)),
\end{equation}
where $\mu_k,\alpha_k \in \Ki_k^{\times}$, $r_k\in \Ni$ and
$S_k \in \Ki_k[[x]]$ satisfies $\val(S_k)\le r_k$. Following
\cite{PoWe17}, we call the parametrisation
\[
(\mu_k T^{e_k},S_k(T)):=\pi_k(T,0)
\]
a pseudo-rational Puiseux expansion (pseudo-RPE for short). Its ring
of definition equals the current residue ring $\Ki_k$, which is a
reduced zero-dimensional $\Ki$-algebra of degree $f_k$ over
$\Ki$. When $F$ is irreducible, the $\Ki_k$'s are fields and the
parametrisation $\pi_k(T,0)$ allows to compute the Puiseux series of
$F$ truncated up to precision $\frac{r_k}{e_k}$, which increases with
$k$ \cite[Section 3.2]{PoWe17}. We show here that the same conclusion
holds when $F$ is pseudo-irreducible, taking care of possible
zero-divisors in $\Ki_k$. To this aim, we prove by induction an
explicit formula for $\pi_k(T,0)$. We need further notations.

\textit{Exponents data.} For all $0\le i\le k \le g$, we define
$Q_{k,i}=q_{i+1}\cdots q_k$ with convention $Q_{k,k}=1$ and let
\[
  B_{k,i}=m_{1} Q_{k,1}+\cdots+m_i Q_{k,i}
\]
with convention $B_{k,0}=0$. Note that $Q_{i,0}=e_i$,
$Q_{g,i}=\hat{e}_i$ and $B_{g,i}=B_i$ for all $i\leq g$. We have the
relations $Q_{k+1,i}=q_{k+1}Q_{k,i} $ and $B_{k+1,i}=q_{k+1} B_{k,i}$
for all $i\le k$ and $B_{k+1,k+1}=q_{k+1} B_{k,k} + m_{k+1}$.

\textit{Coefficients data.} For all $0\le i\le k \le g$, we define
$ \mu_{k,i}:=z_{i+1}^{t_{i+1}Q_{i,i}}\cdots z_{k}^{t_{k}Q_{k-1,i}} $
with convention $\mu_{k,k}=1$ and let
\[
  \alpha_{k,i}:=\mu_{k,1}^{m_1}\cdots \mu_{k,i}^{m_i},
\]
with convention $\alpha_{k,0}=1$.  We have
$\mu_{k+1,i}=\mu_{k,i}z_{k+1}^{t_{k+1} Q_{k,i}}$ and
$\alpha_{k+1,i}=\alpha_{k,i}z_{k+1}^{t_{k+1}B_{k,i}}$ for all
$1\le i\le k$, and $\alpha_{k+1,k+1}=\alpha_{k+1,k}$.

\begin{lem}\label{lem:pik}
  Let $z_0=0$ and $s_0=1$. For all $k=0,\ldots,g$, we have the formula
  \[
  \pi_k(x,y)=\left(\mu_{k,0} x^{e_k}, \sum_{i=0}^{k} \alpha_{k,i}
  x^{B_{k,i}}\left(z_i^{s_i}+c_i\left(\mu_{k,i}
  x^{Q_{k,i}}\right)\right)+\alpha_{k,k} x^{B_{k,k}} y\right).
  \]
\end{lem}
\begin{proof}
  This is correct for $k=0$: the formula becomes
  $\pi_0(x,y)=(x,y+c_0(x))$. For $k>0$, we conclude by induction,
  using the recursives relations for $B_{k,i}$, $\mu_{k,i}$ and
  $\alpha_{k,i}$ above with the definition
  $\pi_{k}(x,y) = \pi_{k-1}(z_{k}^{t_k} x^{q_k} , x^{m_k}
  (z_k^{s_k}+c_{k}(x)+y))$.
\end{proof}

Given $\alpha$ an element of a ring $\Li$, we denote by $\alpha^{1/e}$
the residue class of $Z$ in $\Li[Z]/(Z^e-\alpha)$. For all
$k=0,\ldots,g$, we define the ring extension
\[
  \Li_k:=\Ki_k\big[z_1^{\,\frac{1}{e}},\ldots,z_k^{\,\frac{1}{e}}\big].
\]
Note that $\Li_0=\Ki$. Moreover, since $z_k^{1/e}$ has degree
$e\ell_k>1$ over $\Li_{k-1}$, the natural inclusion
$\Li_{k-1}\subset \Li_{k}$ is \emph{strict}.

\begin{rem}\label{rem:zerodiv}
  Note that $\theta_k:=\mu_{k,0}^{-1/e_k}$ is a well defined
  invertible element of $\Li_k$ (use Remark \ref{rem:q=1}), which by
  Lemma \ref{lem:pik} plays an important role in the connections
  between pseudo-RPE and Puiseux series (proof of Proposition
  \ref{prop:CPE} below). In fact, we could replace $\Li_k$ by the subring $\Ki_k[\theta_k]$ of sharp degree $e_k f_k$ over
  $\Ki$, see \cite{PoWeBis19}. We use $\Li_k$ for convenience,
  especially since $z_k^{1/q_k}$ might not lie in $\Ki_k[\theta_k]$.
  The key points are: $\theta_k\in \Li_k$ and the inclusion
  $\Li_{k-1}\subset \Li_k$ is strict.
\end{rem}

\begin{prop}\label{prop:CPE}
  Let $F\in \Ki[[x]][y]$ be Weierstrass and consider
  $\tilde{S}:=S(\mu^{-1/e}T)$, where $(\mu\,
  T^e,S(T)):=\pi_g(T,0)$. We have
  \[
  \tilde{S}(T)=\sum_{B > 0} a_B T^{B}\in \Li_g[[T]],
  \]
  where $gcd(B_0,\ldots,B_k)|B$ and $a_B\in \Li_k$ for all
  $B< B_{k+1}$ (with convention $B_{g+1}:=+\infty$). Moreover, we have
  for all $1\le k \le g$
  \begin{equation}\label{eq:aBk}
    a_{B_k}=
    \begin{cases}
      \varepsilon_{k} z_k^{\frac{1}{q_k}} \quad 
      \qquad \text{if}\quad  q_k>1 \\%
      \varepsilon_{k} z_k +\rho_k \quad \,\, \,\text{if} \quad q_k=1 
    \end{cases}
  \end{equation}
  where $\varepsilon_{k}\in\Li_{k-1}^\times$ and
  $\rho_{k}\in \Li_{k-1}$. In particular
  $a_{B_k}\in \Li_{k}\setminus \Li_{k-1}$.
\end{prop}

\begin{proof} Note first that $\mu=\mu_{g,0}$ by Lemma \ref{lem:pik},
  so that $\theta_g:=\mu^{-1/e}$ is a well defined invertible element
  of $\Li_g$ (Remark \ref{rem:zerodiv}). In particular,
  $\tilde{S}\in \Li_g[[T]]$ as required. Lemma \ref{lem:pik} applied
  at rank $k=g$ gives
  \begin{equation}\label{eq:SgT}
  S(T)=\sum_{k=0}^{g} \alpha_{g,k}
  T^{B_k}\left(z_k^{s_k}+c_k\left(\mu_{g,k}
  T^{\hat{e}_k}\right)\right).
  \end{equation}
  Denote by $\theta_k:=\mu_{k,0}^{-1/e_k}\in \Li_k^\times$ (Remark
  \ref{rem:zerodiv}). Using the definitions of $\mu_{g,k}$ and
  $\alpha_{g,k}$, a straightforward computation gives
  \begin{equation}\label{eq:alpha}
    \mu_{g,k}\,\theta_g^{\,\hat{e}_{k}}=\theta_k\in \Li_k\quad \textrm{and}\quad
    \alpha_{g,k}\,\theta_g^{\,B_{k}}=\prod_{j=1}^k 
    \left(\mu_{g,j}\theta_{g}^{\hat{e}_{j}}\right)^{m_j} =
    \prod_{j=1}^k \theta_j\in\Li_k.
  \end{equation}
  Combining \eqref{eq:SgT} and \eqref{eq:alpha}, we deduce that
  $\tilde{S}(T)=S(\theta_g T)$ may be written as
  \begin{equation}\label{eq:CPE}
    \tilde{S}(T)=\sum_{k=0}^{g} U_k(\theta_k T^{\hat{e}_k}) \, T^{B_{k}},
    \quad U_k(T):=\left(z_k^{s_k} +c_k(T)\right)\prod_{j=1}^k \theta_j\in \Li_k[[T]].
  \end{equation}
  As $\hat{e}_k=gcd(B_0,\ldots,B_k)$ divides both $\hat{e}_i$ and
  $B_i$ for all $i\le k$, this forces $gcd(B_0,\ldots,B_k)$ to divide
  $B$ for all $B< B_{k+1}$. In the same way, as $\Li_i\subset \Li_k$
  for all $i\le k$, we get $a_B\in \Li_k$ for all $B< B_{k+1}$. There
  remains to show \eqref{eq:aBk}. As $c_k(0)=0$, we deduce that
  \begin{equation}\label{eq:Uk0}
    U_k(0)=z_k^{s_k}\prod_{j=1}^k \theta_j^{m_j}=\varepsilon_k z_k^{1/q_k}
    \quad \textrm{with}\quad \varepsilon_k:=\prod_{j=1}^{k-1} \theta_j\,
    z_j^{\frac{-t_j\,m_k}{q_j\cdots{}q_k}}\in \Li_{k-1},
  \end{equation}
  the second equality using the B\'ezout relation
  $s_k q_k - t_k m_k=1$. Note that $\varepsilon_k\in \Li_{k-1}^\times$
  (Remarks \ref{rem:q=1} and \ref{rem:zerodiv}). Let $\rho_k$ be the
  sum of the contributions of the terms
  $T^{B_i}U_i(\theta_i T^{\hat{e}_i})$, $i\ne k$ to the coefficient
  $T^{B_k}$ of $\tilde S$. So $a_{B_k}=U_k(0)+\rho_k$. As
  $B_1\le\cdots\le B_g$ and $k\ge 1$, we deduce that if
  $U_i(\theta_i T^{\hat{e}_i})\,T^{B_i}$ contributes to $T^{B_k}$,
  then $i<k$ so that
  $U_i(\theta_i T^{\hat{e}_i})\,T^{B_i}\in
  \Li_{k-1}[[T^{\hat{e}_{k-1}}]]$.  We deduce that
  $\rho_k \in \Li_{k-1}$. Moreover, $\rho_k\ne 0$ forces
  $\hat{e}_{k-1}|B_k$. Since $m_{k}$ is coprime to $q_{k}$, we deduce
  from \eqref{eq:BkMk} that $q_k=1$.
\end{proof}
\begin{rem}\label{rmk:Puiseux}
  In contrast to the Newton-Puiseux type algorithms of \cite{PoWe17}
  which compute $\sum_B a_B T^B$ (up to some truncation bound),
  algorithm \PIrr{} of Section \ref{ssec:algo-irr} only allows to
  compute $(a_{B_k}-\rho_k) T^{B_k}$, $k=0,\ldots,g$ in terms of the
  edge data thanks to \eqref{eq:aBk} and \eqref{eq:Uk0}. As shown in
  this section, this is precisely the minimal information required to
  test pseudo-irreducibility and compute the equi-singularity
  type. For instance, the Puiseux series of $F=(y-x-x^2)^2-2x^4$ are
  $S_1=T+T^2(1-\sqrt{2})$ and $S_2=T+T^2(1+\sqrt{2})$ and we only
  compute here the "separating" terms $-\sqrt{2} T^2$ and
  $\sqrt{2} T^2$. Computing all terms of the singular part of the
  Puiseux series of a (pseudo)-irreducible polynomial in quasi-linear
  time remains an open challenge.
\end{rem}
Let us denote by $W\subset \algclos{\Ki}^g$ the zero locus of the
polynomial system defined by the canonical liftings of
$P_1,\ldots,P_g$ in $\Ki[Z_1,\ldots,Z_g]$. Note that $\Card(W)=f$.

Given $\zeta=(\zeta_1,\ldots,\zeta_g)\in W$, the choice of some
$e^{th}$-roots $\zeta_1^{1/e},\ldots,\zeta_g^{1/e}$ in $\algclos{\Ki}$
induces a natural evaluation map
\[
  ev_{\zeta}:\Li_g\simeq\Ki\big[z_1^{\,\frac{1}{e}},\ldots,z_g^{\,\frac{1}{e}}\big]\longrightarrow
  \algclos{\Ki}
\]
and we denote for short $a(\zeta)\in \algclos{\Ki}$ the evaluation of
$a\in \Li_g$ at $\zeta$. There is no loss to assume that when
$\zeta,\zeta'\in W$ satisfy $\zeta_k=\zeta_k'$, we choose
$\zeta_k^{1/e}=\zeta_k'^{1/e}$. We thus have
\begin{equation}\label{eq:hyp}
  (\zeta_1,\ldots,\zeta_k)=(\zeta'_1,\ldots,\zeta'_k)\quad \Longrightarrow
  \quad  a(\zeta)=a(\zeta')\quad \forall\,\,a\in \Li_k.
\end{equation}
The following lemma is crucial for our purpose.
\begin{lem}\label{lem:abzeta}
  Let us fix $\omega$ such that $\omega^e=1$ and let
  $\zeta,\zeta'\in W$. For all $k=0,\ldots,g$, the following
  assertions are equivalent:
  \begin{enumerate}
  \item $a_{B}(\zeta)=a_{B}(\zeta')\omega^{B}$ for all $B\le B_k$.
  \item $a_{B}(\zeta)=a_{B}(\zeta')\omega^{B}$ for all $B< B_{k+1}$.
  \item $(\zeta_1,\ldots,\zeta_k)=(\zeta_1',\ldots,\zeta_k')$ and
    $\omega^{\hat{e}_k}=1$.
  \end{enumerate}
\end{lem}
\begin{proof}
  By Proposition \ref{prop:CPE}, we have $a_B\in \Li_k$ and
  $\hat{e}_k|B$ for all $B<B_{k+1}$ from which we deduce
  $3)\Rightarrow 2)$ thanks to hypothesis (\ref{eq:hyp}). As
  $2)\Rightarrow 1)$ is obvious, we need to show $1)\Rightarrow 3)$.
  We show it by induction on $k$. If $k=0$, the claim follows
  immediately since $\hat{e}_0=e$. Suppose $1)\Rightarrow 3)$ holds
  true at rank $k-1$ for some $k\ge 1$. If
  $a_{B}(\zeta)=a_{B}(\zeta')\omega^{B}$ for all $B\le B_k$, then this
  holds true for all $B\le B_{k-1}$. As
  $\varepsilon_k\in\Li_{k-1}^\times$ and $\rho_k\in\Li_{k-1}$, the
  induction hypothesis combined with (\ref{eq:hyp}) gives
  $\varepsilon_k(\zeta)=\varepsilon_k(\zeta')\ne 0$ and
  $\rho_k(\zeta)=\rho_k(\zeta')$. We use now the assumption
  $a_{B_k}(\zeta)=a_{B_k}(\zeta')\omega^{B_k}$. Two cases occur:
  \begin{itemize}
  \item If $q_k>1$, we deduce from (\ref{eq:aBk}) that
    $\zeta_k^{1/q_k}=\zeta_k'^{1/q_k}\omega^{B_k}$, so that
    $\zeta_k=\zeta_k'\omega^{q_k B_k}$.  As $\hat{e}_{k-1}$ divides
    $q_k B_k$ and $\omega^{\hat{e}_{k-1}}=1$ by induction hypothesis,
    we deduce $\zeta_k=\zeta_k'$, as required. Moreover, we get
    $a_{B_k}(\zeta_k)=a_{B_k}(\zeta_k')$ thanks to (\ref{eq:hyp}), so
    that $\omega^{B_k}=1$.
  \item If $q_k=1$, we deduce from (\ref{eq:aBk}) that
    $\zeta_k+\rho_k(\zeta)=\omega^{B_k}(\zeta_k'+\rho_k(\zeta'))$.  As
    $q_k=1$ implies $\hat{e}_{k-1}=\hat{e}_k|B_k$, we deduce again
    $\omega^{B_k}=1$ and $\zeta_k=\zeta'_k$.
  \end{itemize}
  As $B_k=\sum_{s\le k} m_s \hat{e}_s$, induction hypothesis gives
  $(\omega^{\hat{e}_k})^{m_k}=1$. Since $m_k$ is coprime to $q_k$ and
  $(\omega^{\hat{e}_k})^{q_k}=\omega^{\hat{e}_{k-1}}=1$, this forces
  $\omega^{\hat{e}_k}=1$.
\end{proof}
Finally, we can recover all the Puiseux series of a pseudo-irreducible
polynomial from the parametrisation $\pi_g(T,0)$, as required. More
precisely :
\begin{cor}\label{cor:branches}
  Suppose that $F$ is pseudo-irreducible and Weierstrass. Then $F$
  admits exactly $f$ distinct monic irreducible factors
  $F_{\zeta}\in \algclos{\Ki}[[x]][y]$ indexed by $\zeta\in W$.  Each
  factor $F_{\zeta}$ has degree $e$ and defines a branch with
  classical Puiseux parametrisations $(T^e,\tilde{S}_{\zeta}(T))$
  where
  \begin{equation}\label{eq:SetParam}
    \tilde{S}_{\zeta}(T)=\sum_B a_{B}(\zeta)T^B.
  \end{equation} 
  The $e$ Puiseux series of $F_{\zeta}$ are given by
  $\tilde{S}_{\zeta}(\omega x^{\frac{1}{e}})$ where $\omega$ runs over
  the $e^{th}$-roots of unity and this set of Puiseux series does not
  depend of the choice of the $e^{th}$-roots
  $\zeta_1^{1/e},\ldots,\zeta^{1/e}$.
\end{cor}
\begin{proof}
  As $F$ is pseudo-irreducible, $H_g=y$ (see Section
  \ref{ssec:pseudoIrr}) and $\pi_g^* F(x,0) = 0$ by
  \eqref{eq:pikHk}. We deduce $F(T^e,\tilde{S}_{\zeta}(T))=0$ for all
  $\zeta\in W$.  By \eqref{eq:aBk}, we have $a_{B_k}(\zeta)\ne 0$ for
  all $k$ such that $q_k>1$. Since
  $gcd(B_0=e,B_k \,|\, q_k>1)) = gcd(B_0,\ldots,B_g)=\hat{e}_g=1$, the
  parametrisation $(T^e,\tilde{S}_{\zeta}(T))$ is primitive, that is
  the greatest common divisor of the exponents of the series $T^e$ and
  $\tilde{S}_{\zeta}(T)$ equals one. Hence, this parametrisation
  defines a branch $F_{\zeta}=0$, where
  $F_{\zeta}\in \algclos{\Ki}[[x]][y]$ is an irreducible monic factor
  of $F$ of degree $e$. Thanks to Lemma \ref{lem:abzeta}, these $f$
  branches are distinct when $\zeta$ runs over $W$. As $\deg(F)=e\,f$,
  we obtain in such a way all irreducible factors of $F$. Considering
  other choices of the $e^{th}$ roots of the $\zeta_k$'s would lead to
  the same conclusion by construction, and the last claim follows
  straightforwardly.
\end{proof}
%ICI
\subsection{Pseudo-irreducible implies balanced}\label{ssec:PIrrToEqSing}
\begin{prop}\label{cor:charac}
  Let $F\in \Ki[[x]][y]$ be pseudo-irreducible. Then each branch
  $F_{\zeta}$ of $F$ has characteristic exponents
  $(B_0;B_k \, |\, q_k >1), \,k=1,\ldots,g)$.
\end{prop}
\begin{proof} Thanks to Corollary \ref{cor:branches}, all polynomials
  $F_{\zeta}$ have same first characteristic exponent $B_0=e$. We also
  showed in the proof of Corollary \ref{cor:branches} that
  $a_{B_k}(\zeta)\ne 0$ for all $k\ge 1$ such that $q_k>1$. We
  conclude by Proposition \ref{prop:CPE}.
\end{proof}
\begin{prop}\label{prop:series}
  Let $F\in \Ki[[x]][y]$ be pseudo-irreducible with at least two
  branches $F_{\zeta}, F_{\zeta'}$. We have
  \[
  (F_{\zeta},F_{\zeta'})_0=M_{\kappa},\quad \kappa:=\min
  \left\{k=1,\ldots,g \, |\, \zeta_k \ne \zeta'_k\right\}.
  \]
  and this value is reached exactly
  $\hat{f}_{\kappa-1}-\hat{f}_{\kappa}$ times when $\zeta'$ runs over
  the set $W\setminus\{\zeta\}$.
\end{prop}
\begin{proof} Noether-Merle's formula \eqref{eq:NMformula} combined
  with Proposition \ref{cor:charac} gives
  \begin{equation}\label{eq:merle}
    (F_{\zeta},F_{\zeta'})_0=\sum_{k \le K} 
    (\hat{e}_{k-1}-\hat{e}_k) B_k +  \hat{e}_K \Cont(F_{\zeta},F_{\zeta'})
  \end{equation}
  with $K = \max\{k\,|\,\Cont(F_\zeta,F_{\zeta'}) \geq B_k\}$. Note
  that the $B_k$'s which are not characteristic exponents do not
  appear in the first summand of formula (\ref{eq:merle}) ($q_k=1$
  implies $\hat{e}_{k-1}-\hat{e}_k=0$). It is a classical fact that we
  can fix any root $y$ of $F$ for computing the contact order in
  formula \eqref{eq:ContFormula} (see e.g. \cite[Lemma
  1.2.3]{Ga95}). Combined with Corollary \ref{cor:branches}, we obtain
  the formula
  \begin{equation}\label{eq:cont}
    \Cont(F_{\zeta},F_{\zeta'})=\max_{\omega^e=1}\left(v_T\left(\tilde{S}_{\zeta}(T)
    -\tilde{S}_{\zeta'}(\omega T)\right)\right).
  \end{equation}
  We deduce from Lemma \ref{lem:abzeta} that
  \[
  v_T\left(\tilde{S}_{\zeta}(T)-\tilde{S}_{\zeta'}(\omega
    T)\right)=B_{\bar\kappa},\quad \bar\kappa:=\min \left\{k=1,\ldots,g
    \,\, |\,\, \zeta_k \ne \zeta'_k \,\, \textrm{or}\,\,
    \omega^{\hat{e}_k}\ne 1\right\}.
  \]
  As $\omega=1$ satisfies $\omega^{\hat{e}_k}=1$ for all $k$, we
  deduce from the last equality that the maximal value in
  \eqref{eq:cont} is reached for $\omega=1$ (it might be reached for
  other values of $\omega$). It follows that
  $\Cont(F_{\zeta},F_{\zeta'})=B_{\kappa}$ with
  $\kappa=\min \left\{k \, |\, \zeta_k \ne \zeta'_k\right\}$. We thus
  have $K=\kappa$ and \eqref{eq:merle} gives
  $(F_{\zeta},F_{\zeta'})_0=\sum_{k=1}^{\kappa}(\hat{e}_{k-1}-\hat{e}_k)
  B_k + \hat{e}_{\kappa} B_{\kappa}=M_{\kappa}$,
  the last equality by \eqref{eq:Mk}. Let us fix $\zeta$. As said
  above, we may choose $\omega=1$ in \eqref{eq:cont}. We have
  $v_T(\tilde{S}_{\zeta}(T)-\tilde{S}_{\zeta'}(T))=B_{\kappa}$ if and
  only if $\zeta_k'=\zeta_k$ for $k<\kappa$ and
  $\zeta_{\kappa}\ne \zeta'_{\kappa}$. This concludes, as the number
  of possible such values of $\zeta'$ is precisely
  $\hat{f}_{\kappa-1}-\hat{f}_{\kappa}$.
\end{proof}
If $F$ is pseudo-irreducible, then it is balanced and satisfies both
items of Theorem \ref{thm:pseudo} thanks to Propositions
\ref{cor:charac} and \ref{prop:series}. There remains to show the
converse.

\subsection{Balanced implies pseudo-irreducible}\label{ssec:EqSingToPIrr}
We need to show that $N_g=1$ if $F$ is balanced. We denote more simply
$H:=H_g\in \Ki_g[[x]][y]$, and $\pi_g(T,0)=(\mu T^e,S(T))$. We denote
$H_{\zeta}, S_{\zeta},\mu_{\zeta}$ the images of $H, S,\mu$ after
applying (coefficient wise) the evaluation map
$ev_{\zeta}:\Ki_g\to \algclos{\Ki}$. In what follows, irreducible
means absolutely irreducible.
\begin{lem}\label{lem:samedeg}
  Suppose that $F$ is balanced. Then all irreducible factors of all
  $H_{\zeta}$, $\zeta\in W$ have same degree.
\end{lem}
\begin{proof}
  Let $\zeta\in W$ and let $y_\zeta$ be a root of $H_{\zeta}$. As
  $H_{\zeta}$ divides $(\pi_{g}^* F)_\zeta$ by \eqref{eq:pikHk}, we
  deduce from Lemma \ref{lem:pik} (remember $B_{gg}=B_g$) that
  \[
  F(\mu_{\zeta} x^e,S_{\zeta}(x)+x^{B_g} y_\zeta(x))=0.
  \]
  Hence,
  $y_0(x) := \tilde{S}_{\zeta}(x^{\frac{1}{e}}) +
  \mu_{\zeta}^{-\frac{B_g}{e}}x^{\frac{B_{g}}{e}} \,
  y_\zeta(\mu_{\zeta}^{-\frac{1}{e}}x^{\frac{1}{e}})$
  is a root of $F$ and we have moreover the equality
  \begin{equation}\label{eq:y}
    \deg_{\algclos{\Ki}((x))}(y_0)=e\deg_{\algclos{\Ki}((x))}(y_\zeta),
  \end{equation}
  where we consider here the degrees of $y_0$ and $y_{\zeta}$ seen as
  algebraic elements over the field $\algclos{\Ki}((x))$.  As $F$ is
  balanced, all its irreducible factors - hence all its roots - have
  same degree. Combined with \eqref{eq:y}, this implies that all roots
  - hence all irreducible factors - of all $H_{\zeta}$, $\zeta\in W$
  have same degree.
\end{proof}
\begin{cor}\label{cor:Hg0}
  Suppose $F$ balanced and $N_g>1$. Then there exists some coprime
  positive integers $(q,m)$ and $Q\in \Ki_g[Z]$ monic with non zero
  constant term such that $H$ has lower \edgepoly{}
  \[
  \bar{H}(x,y)=Q\left(y^{q}/x^{m}\right)x^{m \deg(Q)}.
  \]
\end{cor}
\begin{proof}
  As $N_g>1$, the Weierstrass polynomial $H=H_{g}$ is not
  pseudo-degenerated and admits a lower slope $(q,m)$ (we can not have
  $H_g=y^{N_g}$ as $F$ would not be square-free). Hence, its lower
  \edgepoly{} may be written in a unique way
  \begin{equation}\label{eq:Hg0}
    \bar{H}(x,y)=y^r \tilde{Q}\left(y^{q}/x^{m}\right)x^{m \deg(\tilde{Q})}
  \end{equation}
  for some non constant monic polynomial $\tilde{Q}\in \Ki_g[Z]$ with
  non zero constant term and some integer $r\ge 0$.  If $r=0$, we are
  done, taking $Q=\tilde{Q}$. Suppose $r>0$. Let $\zeta\in W$ such
  that $\tilde{Q}_{\zeta}(0)\ne 0$. Applying $ev_{\zeta}$ to
  \eqref{eq:Hg0}, we deduce that $\NP(H_{\zeta})$ has a vertice of
  type $(r,i)$, $0< r < d$ from which follows the well-known fact that
  $H_{\zeta}=A B\in \algclos{\Ki}[[x]][y]$, with $\deg(A)=r$ and
  $\deg(B)=q\deg(\tilde{Q})$. By Lemma \ref{lem:samedeg}, this forces
  $q$ to divide $r$. Hence $r=nq$ for some $n\in \Ni$ and the claim
  follows by taking $Q(Z)=Z^n\tilde{Q}(Z)$.
\end{proof}
\begin{lem}\label{lem:Gparam}
  Suppose $F$ balanced and $N_g>1$. We keep notations of Corollary
  \ref{cor:Hg0}. Let $G$ be an irreducible factor of $F$ in
  $\algclos\Ki[[x]][y]$. Then $e\,q$ divides $n:=\deg(G)$ and there exists
  a unique $\zeta\in W$ and a unique root $\alpha$ of $Q_{\zeta}$ such
  that $G$ admits a parametrisation $(T^n,S_G(T))$, where
  \begin{equation}\label{eq:SG}
    S_G(T)\equiv \tilde{S}_{\zeta}(T^{\frac{n}{e}}) + \alpha^{\frac{1}{q}}
    \mu_{\zeta}^{-\frac{B_g}{e}}T^a \mod \,T^{a+1},
  \end{equation}
  with $a=\frac{n}{e}B_g+\frac{nm}{eq}\in\Ni$, $\alpha^{1/q}$ an
  arbitrary $q^{th}$-root of $\alpha$. Conversely, given $\zeta\in W$
  and $\alpha$ a root of $Q_{\zeta}$, there exists at least one
  irreducible factor $G$ for which \eqref{eq:SG} holds.
\end{lem}
\begin{proof}
  Let $y_{\zeta}^{(i)}$, $i=1,\ldots,N_g$ be the roots of
  $H_{\zeta}$. Following the proof of Lemma \ref{lem:samedeg}, we know
  that each root $y_{\zeta}^{(i)}$ gives rise to a family of
  $e$ roots of $F$
  \begin{equation}\label{eq:yzetai}
  y_{\zeta,\omega}^{(i)}:=\tilde{S}_{\zeta}(\omega
  x^{\frac{1}{e}})+\omega^{B_g}\mu_{\zeta}^{-\frac{B_g}{e}}
  x^{\frac{B_{g}}{e}}y_{\zeta}^{(i)}(\omega\mu_{\zeta}^{-\frac{1}{e}}x^{\frac{1}{e}}),
  \qquad \omega^e=1.
  \end{equation}
  As $H_{\zeta}$ has distinct roots and
  $\tilde{S}_{\zeta}(\omega x^{1/e})\ne \tilde{S}_{\zeta'}(\omega'
  x^{1/e})$ when $(\zeta,\omega)\ne (\zeta',\omega')$ (Lemma
  \ref{lem:abzeta}), we deduce that the $efN_g=\deg(F)$ Puiseux series
  $y_{\zeta,\omega}^{(i)}$ are distinct, getting all roots of $F$. Let
  $G$ be an irreducible factor of $F$ vanishing say at
  $y_0=y_{\zeta,\omega}^{(i)}$. The roots of $G$ are $y_0(\omega' x)$,
  $\omega'^n=1$, where
  $n:=\deg(G)=\deg_{\algclos{\Ki}((x))} (y_{\zeta,\omega}^{(i)})$. As
  $e$ divides $n$ (use \eqref{eq:y}), it follows from
  \eqref{eq:yzetai} that $G$ vanishes at $y_{\zeta,1}^{(i)}$, hence
  admits a parametrisation $(T^n,S_G(T))$, where
  $S_G(T):=y_{\zeta,1}^{(i)}(T^n)$. Corollary \ref{cor:Hg0} ensures
  that $y_{\zeta}^{(i)}(x)= \alpha^{1/q} x^{m/q}+h.o.t.$ for some
  uniquely determined root $\alpha$ of $Q_{\zeta}$. Combined with
  \eqref{eq:yzetai}, we get the claimed formula for $S_G$. Conversely,
  if $\zeta\in W$ and $Q_{\zeta}(\alpha)=0$, there exists at least one
  root $y_{\zeta}^{(i)}$ of $H_{\zeta}$ such that
  $y_{\zeta}^{(i)}(x)= \alpha^{1/q} x^{m/q}+h.o.t$ and by the same
  arguments as above, there exists at least one irreducible factor $G$
  such that \eqref{eq:SG} holds. Finally, since
  $S_G\in \algclos{\Ki}[[T]]$ and since there exists at least one root
  $\alpha\ne 0$ of $Q_{\zeta}$, we must have $nm/eq\in \Ni$. As $e|n$
  and $q$ and $m$ are coprime, we get $eq|n$, as required.
\end{proof}
For a given irreducible factor $G$ of $F$, we denote by
$(\zeta(G),\alpha(G))\in W\times \algclos{\Ki}$ the unique pair
$(\zeta,\alpha)$ such that \eqref{eq:SG} holds. Given $\zeta\in W$,
Corollary \ref{cor:Hg0} and Lemma \ref{lem:Gparam} imply that
\begin{equation}\label{eq:Hbarzeta}
\bar{H}_{\zeta}=\prod_{i
    |\zeta(G_i)=\zeta}(y^q-\alpha(G_i)x^m)^{N(G_i)},
\end{equation}
where $G_1,\ldots, G_{\rho}$ stand for the irreducible factors of $F$
and where $N(G_i):=\deg(G_i)/eq$. Note that by Lemma
\ref{lem:samedeg}, $\deg(G_i)$ and $N(G_i)$ are constant for all
$i=1,\ldots,\rho$.
\begin{cor}\label{cor:expcarac}
  Suppose $F$ balanced and $N_g>1$. Keeping notations as above, the
  lists of the characteristic exponents of the $G_i$'s all begin as
  $\{n\}\cup \{\frac{n}{e}B_k,q_k>1, k=1,\ldots,g\}$. The next
  characteristic exponent is greater or equal than
  $\frac{n}{e}B_g+\frac{nm}{eq}\in \Ni$, with equality if and only if
  $q>1$ and $\alpha(G_i)\ne 0$.
\end{cor}
\begin{proof}
  This follows straightforwardly from Lemma \ref{lem:Gparam} combined
  with Proposition \ref{prop:CPE} (similar argument than for
  Proposition \ref{cor:charac}).
\end{proof}
\begin{cor}\label{cor:interset}
  Suppose $F$ balanced with $N_g>1$. Then
  \[
  (G_i,G_j)_0 > \frac{n^2}{e^2}\left(M_g+\frac{m}{q}\right) \quad
  \Longleftrightarrow\quad
  (\zeta(G_i),\alpha(G_i))=(\zeta(G_j),\alpha(G_j)).
  \]
\end{cor}
\begin{proof}
  Using similar arguments than Proposition \ref{prop:series}, we get
  $\Cont(G_i,G_j)=v_T(S_{G_i}-S_{G_j})$ and we deduce from
  \eqref{eq:SG} and Lemma \ref{lem:abzeta} that
  $\Cont(G_i,G_j) >\frac{n}{e}B_g+\frac{nm}{eq}$ if and only if
  $\zeta(G_i)=\zeta(G_j)$ and $\alpha(G_i)=\alpha(G_j)$. The claim
  then follows from Noether-Merle's formula \eqref{eq:NMformula}
  combined with Corollary \ref{cor:expcarac}.
\end{proof}
\begin{prop}\label{prop:equisingpseudoirr}
  If $F$ is balanced, then it is pseudo-irreducible.
\end{prop}
\begin{proof} We need to show that $N_g=1$. Suppose on the contrary
  that $N_g>1$. We deduce from \eqref{eq:Hbarzeta} that the polynomial
  $Q$ of Corollary \ref{cor:Hg0} satisfies
  \begin{equation}\label{eq:factoQzeta}
    Q_{\zeta}(Z)=\prod_{i |\zeta(G_i)=\zeta}(Z-\alpha(G_i))^{n/eq}
  \end{equation}
  for all $\zeta\in W$.  Let $\alpha$ be a root of $Q_{\zeta}$ and
  $I_{\zeta,\alpha}:=\{i \,|\,
  (\zeta(G_i),\alpha(G_i))=(\zeta,\alpha)\}$.
  Hence, \eqref{eq:factoQzeta} implies that $\alpha$ has multiplicity
  $\frac{n}{eq}\Card(\,I_{\zeta,\alpha})$. As $F$ is balanced, all
  factors have same intersection sets and Corollary \ref{cor:interset}
  implies that all sets $I_{\zeta,\alpha}$ have same cardinality. Thus
  all roots $\alpha$ of all specialisations $Q_{\zeta}$, $\zeta\in W$
  have same multiplicity. In other words, $Q$ is the power of some
  square-free polynomial $P\in \Ki_g[Z]$. If $q=1$, this implies that
  $H=H_g$ is pseudo-degenerated (Definition \ref{def:pseudodeg}),
  contradicting $N_g>1$. If $q>1$, we need to show moreover that $P$
  has invertible constant term. Since there exists at least one non
  zero root $\alpha$ of some $Q_{\zeta}$ (Corollary \ref{cor:Hg0}), we
  deduce from Corollary \ref{cor:expcarac} that at least one factor
  $G_i$ has next characteristic exponent
  $\frac{n}{e}B_g +\frac{nm}{eq}$ (use $q>1$). As $F$ is balanced, it
  follows that all $G_i$'s have next characteristic exponent
  $\frac{n}{e}B_g +\frac{nm}{eq}$, which by Corollary
  \ref{cor:expcarac} forces all $\alpha(G_i)$ - thus all roots
  $\alpha$ of all $Q_{\zeta}$ by last statement of Lemma
  \ref{lem:Gparam} - to be non zero. Thus $P$ has invertible constant
  term and $H=H_g$ is pseudo-degenerated, contradicting $N_g>1$. Hence
  $N_g=1$ and $F$ is pseudo-irreducible.
\end{proof}
The proof of Theorem \ref{thm:pseudo} is complete. $\hfill{\square}$

%%% Local Variables:
%%% mode: latex
%%% TeX-master: "equi-singularity"
%%% End:

\section{A quasi-optimal pseudo-irreducibility test}\label{sec:psi}

Finally, we explain here the main steps of an algorithm which tests
the pseudo-irreducibility of a Weierstrass polynomial and computes its
equisingularity type in quasi-linear time with respect to $\vF$, and
we illustrate it on some examples. Details can be found in
\cite{PoWe19, PoWeBis19}.

\subsection{Computing the lower \edgepoly{}}\label{ssec:phi-main}

We still consider $F\in \Ki[[x]][y]$ a degree $\dy$ square-free
Weierstrass polynomial. In the following, we fix an integer
$0\le k \le g$ and assume that $N_k>1$. For readability, we will omit
the index $k$ for the objects $\Psi,V,\Lambda,\Bc$ introduced below.

Given the edge data $(q_1,m_1,P_1,N_1),\ldots,(q_k,m_k,P_k,N_k)$, we
want to compute $\bar{H}_k$ in quasi-linear time with respect to
$\vF$.

\paragraph{The $(V,\Lambda)$ sequence.} We define recursively two
lists
\[
V=(v_{k,-1},\ldots, v_{k,k})\in \Ni^{k+2} \textrm{ and }
\Lambda=(\lambda_{k,-1},\ldots,\lambda_{k,k})\in \Ki_k^{k+2}.
\]
If $k=0$, we let $V=(1,0)$ and $\Lambda=(1,1)$. Assume $k\ge 1$. Given
the lists $V$ and $\Lambda$ at rank $k-1$ and given the $k$-th edge
data $(q_k,m_k,P_k,N_k)$, we update both lists at rank $k$ thanks to
the formul\ae{}:
\begin{equation}\label{eq:update}
  \begin{cases}
    v_{k,i} = q_k v_{k-1,i} \quad {\small\text{$-1\le i<k-1$}}\\%
    v_{k,k-1} = q_k v_{k-1,k-1}+m_k\\%
    v_{k,k} = q_k\ell_k v_{k,k-1}
  \end{cases}
  \ 
  \begin{cases}
    \lambda_{k,i} = \lambda_{k-1,i}z_k^{t_k v_{k-1,i}}\quad
    {\small\text{$-1\le i<k-1$}}\\%
    \lambda_{k,k-1} = \lambda_{k-1,k-1} z_k^{t_k v_{k-1,k-1}+s_{k}}\\%
    \lambda_{k,k} = q_k
    z_k^{1-s_k-\ell_k}P_k'(z_k)\lambda_{k,k-1}^{q_k
      \ell_k}
  \end{cases}
\end{equation}
where $q_ks_k-m_k t_k=1$, $0\le t_k< q_k$ and $z_k=Z_k\mod P_k$.

\paragraph{Approximate roots and $\Psi$-adic expansion.} Given an
integer $N$ dividing $\dy$, there exists a unique polynomial
$\psi\in \Ki[[x]][y]$ monic of degree $\dy/N$ such that
$\deg(F-\psi^{N}) < \dy-\dy/N$ (see e.g. \cite[Proposition
3.1]{Pop02}). We call it the $N^{th}$ \emph{approximate root} of $F$.
Approximate roots are used in an irreducibility criterion in
$\Ci[[x,y]]$ due to Abhyankhar \cite{Ab89}.

We denote by $\psi_k$ the $N_k^{th}$-approximate root of $F$ and we
let $\psi_{-1}:= x$. We denote
$\Psi = (\psi_{-1} ,\psi_0, \ldots,\psi_k)$ and introduce the set
\begin{equation}\label{eq:Bc}
  \Bc := \{(b_{-1} ,\ldots,b_k)\in \Ni^{k+2} \,\,, \,
  \, b_{i-1}<q_i\,\ell_i\,, i=1,\ldots,k\}.
\end{equation} 
Thanks to the relations $\deg(\psi_i) = \deg(\psi_{i-1})q_i \ell_i$
for all $1 \le i \le k$, an induction argument shows that $F$ admits a
unique expansion
\[
F = \sum_{B\in \Bc} f_B \Psi^B, \quad f_B\in \Ki,
\]
where $\Psi^B:=\prod_{i=-1}^k\psi_i^{b_i}$. We call it the
\emph{$\Psi$-adic expansion} of $F$. We have necessarily $b_k\leq N_k$
while we do not impose any \textit{a priori} condition to the powers
of $\psi_{-1}=x$ in this expansion.

\paragraph{A formula for the lower \edgepoly{}.} For $i\in \Ni$,
we define the integer
\begin{equation}\label{eq:wj}
  w_i := \min \left\{\langle B, V\rangle, \,\, b_k=i,\,\,f_B \ne 0\right\}-v_k(F)
\end{equation}
where $\langle \,,\,\rangle$ stands for the usual scalar product and
with convention $w_i := \infty$ if the minimum is taken over the empty
set. We introduce the set
\[
\Bc(i,w) := \{B \in \Bc(i) \ | \ \, \langle B, V \rangle = w\}
\]
for any $i\in\Ni$ and any $w\in \Ni\cup\{\infty\}$, with convention
$\Bc(i,\infty)=\emptyset$. We get the following key result:

\begin{thm}
  \label{thm:EdgePoly}
  The lower edge $\NP_0(H_k)$ coincides with the lower edge of the
  convex hull of $(i,w_i)_{0\leq i\leq N_k}$. The lower \edgepoly{} of
  $H_k$ equals
  \begin{equation}\label{eq:barHk}
    \bar{H}_k =\sum_{(i,w_i)\in \NP_0(H_k)} \left(\sum_{B\in \Bc(i,w_i+v_k(F))}f_B \Lambda^{B-B_0} \right)x^{w_i} y^i,
  \end{equation}
  where $B_0 := (0,\ldots,0, N_k)$.
\end{thm}

\begin{proof}
  This is a variant of Theorems 2, 3 and 4 in \cite{PoWe19}, where
  degeneracy conditions are replaced now by pseudo-degeneracy
  conditions. The delicate point is that $z_k$ might be here a zero
  divisor when $q_k=1$. However, we can show that we always have
  $\lambda_{kk}\in \Ki_k^\times$. In particular, \eqref{eq:barHk} is
  well-defined and a careful reading shows that the proofs of Theorem
  2, 3 and 4 in \cite{PoWe19} remain valid under the weaker hypothesis
  of pseudo-degeneracy. We refer to Proposition 6 in the longer
  preprint \cite{PoWeBis19} for details.
\end{proof}

\begin{xmp}
  If $F=\sum_{i=0}^\dy a_i y^i$, then $\psi_0=y-c_0(x)$ where
  $c_0=-\frac{a_{\dy-1}}{\dy}$. It follows that at rank $k=0$, the
  coefficients of the $\Psi$-adic expansion of $F$ coincide with the
  coefficients of the $(x,y)$-adic expansion of $H_0$ as defined in
  \eqref{eq:H0}. This illustrates that \eqref{eq:barHk} holds at rank
  $k=0$.
\end{xmp}

\subsection{The algorithm}\label{ssec:algo-irr}
We obtain the following sketch of algorithm. Subroutines \AppRoot{},
\Expand{} and \Edata{} respectively compute the approximate root, the
$\Psi$-adic expansion and the edge data.

\begin{algorithm}[H]
  \nonl\TitleOfAlgo{\PIrr$(F)$\label{algo:irreducible}}%
  \KwIn{$F\in\Ki[[x]][y]$ Weierstrass with $\Char(\Ki)$ not dividing
    $\deg(F)$.}%
  \KwOut{\True{} if $F$ is pseudo-irreducible, and \False{}
    otherwise.}%
  $N\gets \deg(F)$, $V\gets (1,0)$, $\Lambda\gets (1,1)$,
  $\Psi\gets(x)$\;%
  \While{$N > 1$}{%
    $\Psi\gets \Psi\cup \AppRoot{}(F,N)$\;%
    $\sum_{B} f_B \Psi^{B}\gets \Expand{}(F,\Psi)$\;%
    Compute $\bar{H}$ using \eqref{eq:barHk}\;%
    \lIf{$\bar{H}$ is not pseudo-degenerated}{\Return{\False}}%
    $(q,m,P,N)\gets \Edata(\bar{H})$\;%
    Update $V,\Lambda$ using \eqref{eq:update} }%
  \Return{\True{}}
\end{algorithm}

\begin{thm}\label{thm:Irr}
  Algorithm \PIrr{} returns the correct answer.
\end{thm}

\begin{proof}
  Follows from Definition \ref{def:pseudo-irr}, Theorem
  \ref{thm:pseudo} and Theorem \ref{thm:EdgePoly}.
\end{proof}

\paragraph{Proof of Theorem \ref{thm:main2}.} We deduce from
\cite[Proposition 12]{PoWe19} that algorithm \PIrr{} may run with an
expected $\Ot(\vF)$ operations over $\Ki$\footnote{In
  \cite[Prop.12]{PoWe19}, the condition $P_k(0)\in \Ki_k^\times$ is
  imposed even if $q_k=1$, but this has no impact from a complexity
  point of view.}. To this aim, we use:
\begin{itemize}
\item Suitable truncation bounds for the powers of $x$, updated at
  each step.
\item Primitive representation of the various residue rings $\Ki_k$
  (Las-Vegas subroutines)
\item Suitable implementation of subroutines \AppRoot{}, \Expand{}, \Edata{}
  and of the pseudo-degeneracy tests (square-free univariate factorisation over
  direct product of fields, see Remark \ref{rem:testpseudodeg}).
\end{itemize}
If $F$ is pseudo-irreducible, we can deduce from the edge data of $F$
the characteristic exponents and the intersection sets of $F$ (Theorem
\ref{thm:pseudo}), together with the discriminant valuation $\vF$
(Corollary \ref{cor:disc}). Theorem \ref{thm:main2} follows.
$\hfill\square$

\begin{rem}
  Note that if we rather use computations \eqref{eq:Gk} and
  \eqref{eq:Hk} up to suitable precision to check if $F$ is
  pseudo-irreducible (hence balanced), the underlying algorithm has
  complexity $\Ot(\dy\vF)$ when using similar cautious algorithmic
  tricks as above (see \cite[Section 3]{PoWe17}). This bound is sharp
  (see e.g. \cite[Example 1]{PoWe19}) and is too high for our
  purpose. One of the main reason is that computing the intermediate
  polynomials $G_k$ in \eqref{eq:Gk} via Hensel lifting up to
  sufficient precision might cost $\Omega(\dy\vF)$. This shows the
  importance of using approximate roots.
\end{rem}
 
\subsection{Non Weierstrass polynomials.}\label{ssec:nonweirestrass}
From a computational aspect with a view towards factorisation in
$\Ki[[x]][y]$, it seems interesting to extend Theorems \ref{thm:main2}
and \ref{thm:pseudo} to the case of non Weierstrass polynomials.

\paragraph{Non Weierstrass balanced polynomials.} If $F$ is
absolutely irreducible but not necessarily Weierstrass, it defines a
unique germ of irreducible curve on the line $x=0$, with center
$(0,c)$, $c\in \algclos{\Ki}\cup \{\infty\}$. It seems to be a natural
option to require that the equisingularity type of a germ of plane
curve along the line $x=0$ does not depend on its center. This point
of view leads us to define then the characteristic exponents of $F$ as
those of the shifted polynomial $F(x,y+c)$ if $c\in \algclos{\Ki}$ or
of the reciprocal polynomial $\tilde{F}=y^{\dy} F(x,y^{-1})$ if
$c=\infty$ (note that these change of coordinates have not impact on
the tangency with the $x$-axis). The formula \eqref{eq:intmult} of the
intersection multiplicity also extends by linearity to arbitrary
coprime polynomials $G, H\in \algclos{\Ki}[[x]][y]$, taking into
account the sum of intersection multiplicities between all germs of
curves defined by $G$ and $H$ along the line $x=0$. The intersection
might be now zero if (and only if) $G$ and $H$ do not have branches
with the same center. We can thus extend the definition of
intersection sets to non Weierstrass polynomials, allowing now
$0\in \Gamma(F_i)$.  Finally, we may extend Definition
\ref{def:balanced} to an arbitrary square-free polynomial
$F\in \Ki[[x]][y]$.

\paragraph{Pseudo-Irreducibility of non Weierstrass polynomials.} We
distinguish the monic case, for which approximate roots are defined,
and the non monic case.

$\bullet$ If $F$ is monic, the construction of Section
\ref{sec:pseudo-irr} remains valid, a slight difference being that the
first polynomial $H_0$ might be now monic (and $m_1=0$ is
allowed). However the remaining polynomials $H_k$ are still
Weierstrass for $k\ge 1$. Hence the definition of a pseudo-irreducible
polynomial extends to the monic case and we can check that Theorem
\ref{thm:pseudo} still hold for monic polynomials. Moreover, the
approximate root of a monic polynomial $F$ are still defined, and it
is shown in \cite{PoWe19} that Theorem \ref{thm:EdgePoly} holds too in
this case. Hence, we let run algorithm \PIrr{} as in the Weierstrass
case. However to keep a small complexity, we do not compute primitive
elements of $\Ki_k$ over the field $\Ki$ but only over the next
residue ring $\Ki_1=\Ki_{P_1}$. The overall complexity of this slightly
modified algorithm becomes $\Ot(\vF+\dy)$. We refer the reader to
\cite{PoWe19} for details.
  
$\bullet$ There remains to consider the case when $F$ is not
monic. One way to deal with this problem is to use a projective change
of the $y$ coordinates in order to reduce to the monic case. Since
$\Ki$ has at least $\dy+1$ elements by assumption, we can compute
$z\in \Ki$ such that $F(0,z)\ne 0$ with at most $\dy+1$ evaluation of
$F(0,y)$ at $z=0,1,\ldots,\dy$. This costs at most $\Ot(\dy)$ using
fast multipoint evaluation \cite[Corollary 10.8]{GaGe13}. One such a
$z$ is found, we can apply the previous strategy to the polynomial
$\tilde{F}:=y^{\dy}F\left(\frac{zy+1}{y}\right)\in \Ki[[x]][y]$ which
has by construction an invertible coefficient that we simply invert up
to suitable precision. We have $\deg(F)=\deg(\tilde{F})$ and
$\vF(F)=\vF(\tilde{F})$ (assuming that $\vF$ is then defined as the
valuation of the resultant between $F$ and $F_y$ instead of the
valuation of the discriminant which may vary under projective change
of coordinates). So the complexity remains the same.  Moreover, $F$
and $\tilde{F}$ have same number of absolutely irreducible factors,
same sets of characteristic exponents (by the very definition) and
same intersection sets (use that the $x$-valuation of the resultant is
invariant under projective change of the $y$ coordinate (see
e.g. \cite[Chapter 12]{GeKaZe94}). In particular, $F$ is balanced if and
only if $\tilde{F}$ is. This shows that we can test if an arbitrary
square-free polynomial $F$ is balanced - and if so, compute the
equisingular types of all germs of curves it defines along the line
$x=0$ - within $\Ot(\vF+\dy)$ operations over $\Ki$. We refer the
reader to \cite{PoWe19} for details.

\begin{rem}
  If $F$ is not monic, we could also have followed the following
  option. We can extend the construction of Section
  \ref{sec:pseudo-irr} by allowing positive slopes at the first call
  (so $m_1<0$ is allowed) and extend Theorem \ref{thm:pseudo} by
  considering approximate roots in the larger ring
  $\Ki((x))[y]$. However, it turns out that this option is not
  compatible with our $PGL_2(\Ki)$-invariant point of view when $F$
  defines a germ centered at $(0,\infty)$, and Theorem
  \ref{thm:pseudo} would require some slight modifications to hold in
  this larger context.
\end{rem}

\paragraph{Bivariate polynomials.} If the input $F$ is given as a
bivariate polynomial $F\in \Ki[x,y]$ with partial degrees
$\dx:=\deg_x(F)$ and $\dy=\deg_y(F)$, the well known upper bound
$\vF\le 2\dx\dy$ leads to a complexity estimate $\Ot(\dx\dy)$ which
is quasi-linear with respect to the arithmetic size of the
input. Moreover, up to perform a slight modification of the algorithm,
there is no need to assume $F$ square-free in this ``algebraic'' case
(see again \cite{PoWe19} for details).

\subsection{Some examples}\label{ssec:examples}

\begin{xmp}[\textbf{balanced}]\label{ex1}
  Let
  $F=y^6-3x^3 y^4-2x^2 y^4+3x^6 y^2+x^4 y^2-x^9+2x^8-x^7\in \Qi[x,y]$.
  This small example is constructed in such a way that $F$ has $3$
  irreducible factors $(y-x)^2-x^3$, $(y+x)^2-x^3$, $y^2-x^3$ and we
  can check that $F$ is balanced, with $e=2$, $f=3$ and
  $\C(F_i)=(2;3)$ and $\Gamma_i(F)=(4,4)$ for all $i=1,2,3$. Let us
  recover this with algorithm \PIrr{}.

  \emph{Initialise.} We have $N_0=\dy=6$, and we let $\psi_{-1}=x$,
  $V=(1,0)$ and $\Lambda=(1,1)$.

  \emph{Step $0$.} The $6^{th}$-approximate root of $F$ is $\psi_0=y$
  and we deduce that $\bar{H}_0=y^6-2x^2 y^4+x^4 y^2=(y(y^2-x^2))^2.$
  Thus, $H_0$ is pseudo-degenerated with edge data
  $(q_1,m_1,P_1,N_1)=(1,1,Z_1^3-Z_1,2)$.  Accordingly to
  \eqref{eq:update}, we update $V=(1,1,1)$ and
  $\Lambda=(1,z_1,3z_1^2-1)$. Note that $\NP(F)$ is not straight. In
  particular, $F$ is reducible in $\Qi[[x]][y]$.

  \emph{Step $1$.} The $2^{th}$-approximate root of $F$ is
  $\psi_1=y^3-\frac32 x^3y-x^2y$ and $F$ has $\Psi$-adic expansion
  $F=\psi_1^2 -3\psi_0^2 \psi_{-1}^5+\frac34 \psi_0^2
  \psi_{-1}^6-\psi_{-1}^7+2\psi_{-1}^8- \psi_{-1}^9$.
  The monomials reaching the minimal values \eqref{eq:wj} are
  $\psi_1^2$ (for $j=2$) and $-3\psi_0^2 \psi_{-1}^5$ and
  $\psi_{-1}^7$ (for $j=0$). We deduce from \eqref{eq:barHk} that
  $\bar{H}_1=y^2-\alpha x$, where $\alpha=(3z_1^2+1)/(3z_1^2-1)^2$ is
  easily seen to be invertible in $\Qi_1$ (in practice, we compute
  $P\in \Qi[Z_1]$ such that $\alpha=P\mod P_1$ and we check
  $gcd(P_1,P)=1$). We deduce that $H_1$ is pseudo-degenerated with
  edge data $(q_2,m_2,P_2,N_2)=(2,1,Z_2-\alpha,1)$.  As $N_2=1$, we
  deduce that $F$ is balanced with $g=2$.

  \emph{Conclusion.} We deduce from Theorem \ref{thm:pseudo} that $F$
  has $f=\ell_1\ell_2=3$ irreducible factors over
  $\algclos{\Ki}[[x]][y]$ of same degrees $e=q_1 q_2=2$.  Thanks to
  \eqref{eq:BkMk}, we compute $B_0=e=2$, $B_1=2$, $B_2=3$ and $M_1=4$,
  $M_2=6$.  We deduce that all factors of $F$ have same characteristic
  exponents $\C(F_i)=(B_0;B_2)=(2;3)$ and same intersection sets
  $\Gamma_i(F)=(M_1,M_1)=(4,4)$ (i.e. $M_1$ which appears
  $\hat{f}_0-\hat{f}_1=3-1$ times), as
  required. %, and we can check that the result is correct (fortunately !)
\end{xmp}

\begin{xmp}[\textbf{non balanced}]\label{ex2}
  Let
  $F={{y}^{6}}-{{x}^{6}} {{y}^{4}}-2 {{x}^{4}} {{y}^{4}}-2 {{x}^{2}}
  {{y}^{4}}+2 {{x}^{10}} {{y}^{2}}+3 {{x}^{8}} {{y}^{2}}-2 {{x}^{6}}
  {{y}^{2}}+{{x}^{4}} {{y}^{2}}-{{x}^{14}}+2 {{x}^{12}}-{{x}^{10}}\in
  \Qi[x,y]$.
  This second small example is constructed in such a way that $F$ has
  $6$ irreducible factors $y+x-x^2$, $y+x-x^2$, $y-x-x^2$, $y-x+x^2$,
  $y-x^3$ and $y+x^3$ and we check that $F$ is not balanced, as
  $\Gamma_i(F)=(1,1,1,1,2)$ for $i=1,\ldots,4$ while with
  $\Gamma_i(F)=(1,1,1,1,3)$ for $i=5,6$. Let us recover this with
  algorithm \PIrr{}.

  \emph{Initialise.} We have $N_0=\dy=6$, and we let $\psi_{-1}=x$,
  $V=(1,0)$ and $\Lambda=(1,1)$.

  \emph{Step $0$.} The $6^{th}$-approximate root of $F$ is $\psi_0=y$
  and we deduce that $\bar{H}_0=y^6-2x^2 y^4+x^4 y^2=(y(y^2-x^2))^2.$
  Thus, as in Example \ref{ex1}, $H_0$ is pseudo-degenerated with edge
  data $(q_1,m_1,P_1,N_1)=(1,1,Z_1^3-Z_1,2)$.  Accordingly to
  \eqref{eq:update}, we update $V=(1,1,1)$ and
  $\Lambda=(1,z_1,3z_1^2-1)$.

  \emph{Step $1$.} The $2^{th}$-approximate root of $F$ is
  $\psi_1=y^3-yx^2-yx^4-\frac12 yx^{6}$ and $F$ has $\Psi$-adic
  expansion
  $F=\psi_1^2 -\psi_{-1}^{10}+2
  \psi_{-1}^{12}-\psi_{-1}^{14}-4\psi_{-1}^{6}\psi_{0}^{2}+\psi_{-1}^{8}\psi_{0}^{2}+\psi_{-1}^{10}\psi_{0}^{2}-\frac14\psi_{-1}^{12}\psi_{0}^{2}$.
  The monomials reaching the minimal values \eqref{eq:wj} are
  $\psi_1^2$ (for $j=2$) and $-4\psi_{-1}^{6}\psi_{0}^{2}$ (for
  $j=0$). We deduce from \eqref{eq:barHk} that
  $\bar{H}_1=y^2-\alpha x^2$, where $\alpha=4z_1^2/(3z_1^2-1)^2$. As
  $z_1$ is a zero divisor in $\Qi_1=\Qi[Z_1]/(Z_1^3-Z_1)$ and
  $(3z_1^2-1)=P_1'(z_1)$ is invertible in $\Qi_1$, we deduce that
  $\alpha$ is a zero divisor. It follows that $\bar{H}_1$ is not the
  power of a square-free polynomial. Hence $H_1$ is not
  pseudo-degenerated and $F$ is not balanced (with $g=1$), as
  required. In order to factorise $F$, we would need at this stage
  to split the algorithm accordingly to the discovered factorisation
  $P_1=Z_1(Z_1^2-1)$ before continuing the process, as described in
  \cite{PoWe17}.
\end{xmp}

\begin{xmp}[\textbf{non Weierstrass}]\label{ex:nonweierstrass}
  Let
  $F=(y+1)^6-3x^3(y+1)^4-2(y+1)^4+3x^6(y+1)^2+(y+1)^2-x^9+2x^6-x^3$. We
  have $F=((y+2)^2-x^3)((y+1)^2-x^3)(y^2-x^3)$ from which we deduce
  that $F$ is balanced with three irreducible factors with
  characteristic exponents $\C(F_i)=(2,3)$ and intersection sets
  $\Gamma_i(F)=(0,0)$. Let us recover this with algorithm \PIrr{}.

  \emph{Initialise.} We have $N_0=\dy=6$, and we let $\psi_{-1}=x$,
  $V=(1,0)$ and $\Lambda=(1,1)$.

  \emph{Step $0$.} The $6^{th}$-approximate root of $F$ is
  $\psi_0=y+1$. We have
  $F=\psi_0^6-3\psi_{-1}^3 \psi_0^4- 2 \psi_0^4 + 3 \psi_{-1}^6
  \psi_0^2 + \psi_0^2 - \psi_{-1}^9 + 2 \psi_0^6 - \psi_{-1}^3$.
  By \eqref{eq:wj}, the monomials involved in the lower edge of $H_0$
  are $\psi_0^6, -2\psi_0^4, \psi_0^2$. We deduce from
  \eqref{eq:barHk} that $\bar{H}_0=(y^3-y)^2$ so that $H_0$ is
  pseudo-degenerated with edge data
  $(q_1,m_1,P_1,N_1)=(1,0,Z_1^3-Z_1,2)$. Note that $m_1=0$. This is
  the only step of the algorithm where this may occur. Using
  \eqref{eq:update}, we update $V=(1,0,0)$ and
  $\Lambda=(1,z_1,3z_1^2-1)$.
  
  \emph{Step $1$} The $N_1=2^{th}$ approximate root of $F$ is
  $\psi_1=(y+1)^3-3/2 x^3 (y+1)-(y+1)$ and $F$ has $\Psi$-adic
  expansion
  $F=\psi_1^2 -\psi_{-1}^3-3\psi_{-1}^3
  \psi_0^2+2\psi_{-1}^6-\psi_{-1}^9+3/4 \psi_{-1}^6\psi_0^2$.
  We deduce that the monomials reaching the minimal values
  \eqref{eq:wj} are $\psi_1^2$ (for $j=2$) and $-\psi_{-1}^3$,
  $-3\psi_{-1}^3 \psi_0^2$ (for $j=0$). We deduce from
  \eqref{eq:barHk} that $\bar{H}_1=y^2-\alpha x^3$, where
  $\alpha=(\lambda_{1,-1}^3+3\lambda_{1,-1}^3\lambda_{1,0}^2)\lambda_{1,1}^{-2}=(3z_1^2+1)/(3z_1^2-1)^2$
  is easily seen to be invertible in $\Qi_1$. We deduce that $H_1$ is
  pseudo-degenerated with edge data
  $(q_2,m_2,P_2,N_2)=(2,3,Z_2-\alpha,1)$.  As $N_2=1$, we deduce that
  $F$ is balanced with $g=2$. By Theorem \ref{thm:pseudo} (assuming
  only $F$ monic), we get that $F$ has $f=\ell_1\ell_2=3$ irreducible
  factors over $\algclos{\Ki}[[x]][y]$ of same degrees $e=q_1 q_2=2$.
  Thanks to \eqref{eq:BkMk}, we compute $B_0=e=2$, $B_1=0$, $B_2=3$
  and $M_1=0$, $M_2=6$.  By Theorem \ref{thm:pseudo}, we deduce that
  all factors of $F$ have same characteristic exponents
  $\C(F_i)=(B_0;B_2)=(2;3)$ and same intersection sets
  $\Gamma_i(F)=(M_1,M_1)=(0,0)$ as required.
\end{xmp}

%%% Local Variables:
%%% mode: latex
%%% TeX-master: "equi-singularity"
%%% End:

\bibliographystyle{abbrv} {\bibliography{tout}}
\addcontentsline{toc}{section}{References.}

\end{document}